\documentclass[12pt]{amsart}

\usepackage
{geometry}
 
\usepackage{amsmath,amscd,amssymb,amsfonts,latexsym, mathtools,hhline,color,bm,array,booktabs,float}
\usepackage[all, cmtip]{xy}
\usepackage{url}
\definecolor{hot}{RGB}{65,105,225}
 
\usepackage
[pagebackref=true,colorlinks=true]
{hyperref}
\usepackage{cleveref}

\usepackage{ textcomp }
\usepackage{ tipa }

\usepackage{graphicx,enumerate}
\usepackage{enumitem}
\usepackage[normalem]{ulem}
\usepackage{bm}

\newcommand{\TeXheadline}[1]{\texorpdfstring{#1}{Lg}}

\newcommand{\vartilde}[1]{\widetilde{#1}}

\newcommand{\RR}{\mathbb{R}}
\newcommand{\CC}{\mathbb{C}}
\newcommand{\PP}{\mathbb{P}}
\newcommand{\VV}{\mathbb{V}}

\newcommand{\xx}{{\boldsymbol{x}}}
\newcommand{\yy}{{\boldsymbol{y}}}

\newcommand{\bfzero}{{\mathbf{0}}}
\usepackage{etoolbox}
\robustify\bfseries

\DeclareMathOperator{\rank}{rank}
\DeclareMathOperator{\Gr}{Gr}
\newcommand{\rootcount}{\textbf{68\hspace{.07em}774}}

\newcommand{\dimX}{r}

\newcommand{\PPtwo}{\PP^{m}\times \PP^{n}}

\newcommand{\IntersectionWithHypersurface}{{\tt IntersectionWithHypersurface}}
\newcommand{\EliminateRedundantComponents}{{\tt EliminateRedundantComponents}}
\newcommand{\uStartPoints}{{\tt uStartPoints}}
\newcommand{\uMultiprojectiveStartPoints}{{\tt uMultiprojectiveStartPoints}}
\usepackage{siunitx}

\theoremstyle{plain}
\theoremstyle{definition}
\newtheorem{theorem}{Theorem}[section]
\newtheorem{example}{Example}[section]
\newtheorem{proposition}[theorem]{Proposition}

\newtheorem{corollary}[theorem]{Corollary}

\newtheorem{remark}[theorem]{Remark}

\newtheorem{ex}[theorem]{Example}
\newtheorem*{ex*}{Example}

\usepackage{xcolor}  
\newcommand{\defcolor}[1]{{\color{cyan}#1}} 
\newcommand{\demph}[1]{\defcolor{{\sl #1}}} 

\definecolor{fcolor}{RGB}{212, 17, 89}
\newcommand{\colorf}[1]{{\color{fcolor}#1}}
\definecolor{gcolor}{RGB}{26, 133, 255}
\newcommand{\colorg}[1]{{\color{gcolor}#1}}

\newcommand{\blue}[1]{{\color{blue}#1}}
\newcommand{\magenta}[1]{{\color{magenta}#1}}

\newcommand{\epsT}{t}

\usepackage[linesnumbered,ruled,vlined]{algorithm2e}

\title{$u$-generation: solving systems of polynomials equation-by-equation}

\author{Timothy Duff}\thanks{T.D. acknowledges support from an NSF Mathematical Sciences Postdoctoral Research Fellowship (DMS-2103310)}
\address{Department of Mathematics, University of Washington, Box 354350,
Seattle, Washington 98195–4350}
\email {timduff@uw.edu}\urladdr{\url{https://timduff35.github.io/timduff35/}}
\author{Anton Leykin}\thanks{Research of A.L. is supported in part by NSF DMS-2001267.}
\address{School of Mathematics, Georgia Institute of Technology, 686 Cherry Street NW, Atlanta GA 30308}
\email {anton.leykin@gmail.com}\urladdr{\url{https://antonleykin.math.gatech.edu/index.html}}
\author{Jose Israel Rodriguez}\thanks{ Research of J.I.R. is supported by the Office of the Vice Chancellor for Research and Graduate Education at U.W. Madison with funding from the Wisconsin Alumni Research Foundation.}
\address{Department of Mathematics,         University of Wisconsin-Madison,  480 Lincoln Drive, Madison WI 53706-1388, USA.}
\email {jose@math.wisc.edu}\urladdr{\url{https://sites.google.com/wisc.edu/jose/}}
\thanks{All authors thank the Institute for Mathematics and its Applications in Minneapolis for hosting them for a series of ``SageMath and Macaulay2'' meetings in 2019-2020 where the ideas of this work were born.}



\begin{document}


\begin{abstract}  
We develop a new method that improves the efficiency of equation-by-equation algorithms for solving polynomial systems.
Our method is based on a novel geometric construction, and reduces the total number of homotopy paths that must be numerically continued. 
These improvements may be applied to the basic algorithms of numerical algebraic geometry in the settings of both projective and multiprojective varieties.
Our computational experiments demonstrate significant savings obtained on several benchmark systems.
We also present an extended case study on maximum likelihood estimation for rank-constrained symmetric $n\times n$ matrices, in which multiprojective $u$-generation allows us to complete the list of ML degrees for $n\le 6.$
\end{abstract}

\maketitle

Consider a system of homogeneous polynomial equations defining a complex projective variety, also known as the solution set of the system. This is a set of points in the complex projective space that give zero value to every polynomial in the system. Describing this variety is a basic task in computational algebraic geometry and is the primary meaning of \emph{solving} the system.

The machinery of \emph{numerical} algebraic geometry provides a description in terms of \emph{witness sets} (see \Cref{sec:witness-sets}) of equidimensional pieces of the solution set. 
We aim to improve the class of \emph{equation-by-equation} solvers that accomplish this task.

As the name suggests, given a description of the solution set of a system  such algorithms proceed by appending a new polynomial to the system, on one hand, and producing a description of the solution set for the new system, on the other. Geometrically speaking, the new variety is obtained as the intersection of the old variety and the hypersurface where the new polynomial vanishes.
\thispagestyle{empty}
\section{Introduction}

\subsection{History of equation-by-equation solvers}

The approach of~\cite{SV:generic-points-on-algebraic-sets} modifies polynomials in the system by adding linear terms in a set of new ``slack variables'' producing ``embedded systems''. They develop a ``cascade of homotopies between embedded systems'' which, historically, may be considered the first practical equation-by-equation solver in the framework of numerical algebraic geometry.   


The algorithm of~\cite{SVW:equation-by-equation} also introduces new variables at each step of a different equation-by-equation cascade based on \emph{diagonal homotopies}. The number of additional variables in both this method and the method of~\cite{SV:generic-points-on-algebraic-sets} equals at least the dimension of the solution set.  

The ingenuity of the \emph{regeneration}~\cite{HSW:regeneration,HSW:regenerative-cascade,HW:unification} approach stems from a realization that at each step of the cascade, when considering a new polynomial of degree $d$, one can ``replace'' it with a product of random $d$ linear forms. This results in a {\bf two-stage} procedure that, first, precomputes $d$ copies of witness sets corresponding to the linear factors and, second, deforms the union of $d$ hyperplanes into a hypersurface given by the original polynomial. The \emph{homotopy} realizing this deformation describes a family of 0-dimensional varieties in $\PP^n$ --- no extra variables are introduced.

\subsection{Contributions and outline}
We develop a new {\bf one-stage} procedure dubbed \emph{$u$-generation} for a step in the equation-by-equation cascade.
The core is a geometric construction that relies on a homotopy in $\PP^{n+1}$, thus introducing {\bf one} new variable~$u$. In fact, this new variable can be eliminated when implementing the method (see \Cref{rem:eliminate-u}) and then it performs similarly to the second stage of regeneration, thus saving the cost of performing the first stage. 

The rest of the paper is outlined as follows.
\Cref{sec:projective} provides necessary background on witness sets, introduces $u$-generation, and --- for completeness --- outlines a simple algorithm for an equation-by-equation cascade.  
In \Cref{sec:multiprojective} we extend our approach, albeit in a nontrivial way, to multihomogeneous systems: homotopies are transplanted from $\PP^{n_1}\times\dots\times\PP^{n_k}$ to $\PP^{n_1+1}\times\dots\times\PP^{n_k+1}$.   
\Cref{sec:experiments} describes the results of several computational experiments, comparing $u$-generation to regeneration.
In~\ref{subsec:u-Pn}, we apply the methods of~\Cref{sec:projective} to several benchmark problems, demonstrating potential savings brought by $u$-generation.
In~\ref{subsec:MLE}, we demonstrate how $u$-generation in the multiprojective setting may be applied to solve nontrivial problems in maximum-likelihood estimation.
\Cref{sec:conclusion} provides short a conclusion. 


\section{\TeXheadline{$u$}-generation in \TeXheadline{$\PP^n$}}\label{sec:projective}

\subsection{The homotopy}\label{ss:homotopy}
Let $X$ be a closed subvariety of complex projective space $\PP^{n}.$ 
A straight-line homotopy on $X$ has the form
\begin{equation}
\label{eq:Ht-Pn}
H_t = (1-t)F + tG,\ t\in[0,1],
\end{equation}
where $F$ and $G$ consist of $\dimX $-many homogeneous polynomials of matching degrees in indeterminates $[x_0:x_1:\cdots : x_n]$.
We abbreviate the straight-line homotopy~\eqref{eq:Ht-Pn} by $F\leadsto G$.
Given that $\VV (H_t) \cap X$ is a finite set, for $t\in [0,1],$ one may consider homotopy paths emanating from \emph{start points} $\VV(F) \cap X.$
A typical application of numerical homotopy continuation is to track these paths  
in an attempt to compute the \emph{endpoints} $\VV (G) \cap X.$\footnote{In this paper we shall gloss over many details of how homotopy tracking is accomplished in practice and issues arising from the need to use approximations of points.
For this we refer the reader to introductory chapters of \cite{SW05}.
}

Suppose now that $X$ is a curve which is a component or a union of one-dimensional components of $\VV(F)\subset \PP^n.$ 
Consider the cone $\vartilde X \subset \PP^{n+1}$ with coordinates $[u:\xx]=[u:x_0:x_1:\dots:x_n]$. 
In the chart $u=1$, this is the affine cone over $X,$ and 
\begin{equation}\label{eq:tilde-X}
    \vartilde X = \overline{\{ [u:\xx] \mid \xx \in X \} } 
= \{[u:\xx] \mid \xx \in X \} \cup \{ [1:0:\cdots : 0]\}.
\end{equation}
Given $g_1 \in \CC[\xx]_d$ (homogeneous of degree $d$) and $\ell\in\CC[\xx]_1$ we consider a homotopy 
\begin{equation}
\label{eq:Htwave-Pn}
\vartilde H_t : (g_0,\ell)
\leadsto
(g_1,u)\quad \text{ on }\vartilde X,
\end{equation}
where $g_0 \in \CC[u,\xx]_d$.

\begin{proposition}\label{prop:endpoint}
For generic $g_0\in \CC[u,\xx]_d$ and $\ell_0 \in \CC[\xx]_1$, the cardinality of points on $\vartilde X$ satisfying $\vartilde H_t$ is equal to $d\cdot \deg X$ for $t\in[0,1)$, 
where $\deg X$ denotes the degree of the projective variety $X$.

The start points of the homotopy $\vartilde H_t$ are
$$
\VV(\vartilde H_0) \cap \vartilde X = \{ [u:\xx] \mid [\xx] \in X\cap\VV(\ell), \text{ and } u \text{ satisfies } g_0(u,\xx) = 0 \}.
$$
The endpoints of $\vartilde H_t$ lie in the set
$$\{ [0:\xx] \mid [\xx] \in X\cap\VV(g_1)\}.$$
In the case when this set is finite, every point is reached.

\end{proposition}

\begin{proof}
Consider the exceptional set 
\begin{align*}
\Sigma =\{(g, \ell) \in \CC [u, \xx]_d \times \CC [u, \xx]_1 \  \mid\  |\vartilde X \cap \VV (g, \ell )| \neq d\cdot\deg X \} 
\end{align*}
in the affine space of all coefficients of pairs of polynomials in $\CC [u, \xx]_d \times \CC [u, \xx]_1$.
Genericity of $\ell_0$ implies $X \cap \VV (\ell_0)$ contains $\deg X$ points, and so long as the coefficients of $g_0$ lie outside of $(\deg X)$-many hypersurfaces in $\CC^{N}$ we have $(g_0, \ell_0)\notin \Sigma .$
It follows by the usual argument~\cite[Lemma 7.12]{SW05} that the real segment $(1-t)(g_0, \ell) + t(g_1, u)$ for all $t\in [0, 1)$ is also disjoint from $\Sigma .$ 
The description of the start points and where the endpoints lie is clear from the definition of $\vartilde H_t.$

Our last claim follows from the parameter continuation theorem~\cite[Theorem 7.1.6]{SW05}. We give an alternative elementary self-contained proof below.

Consider the incidence correspondence
\[
\Gamma  = \{ \left( (g, \ell), [\vartilde u : \vartilde \xx] \right) \subset \CC [u, \xx]_d \times \CC [u, \xx]_1 \times \vartilde X \mid g([\vartilde u :\vartilde \xx ] )= \ell ([ \vartilde u : \vartilde \xx ]) = 0 \},
\]
equipped with the projection $\pi : \Gamma \to \CC [u, \xx]_d \times \CC [u, \xx]_1 .$
The map $\pi$ is a branched covering: restricted to the preimage of the complement of $\Sigma$, it is a topological covering map both in usual (complex) and Zariski topology.

Suppose the fiber $\pi^{-1} (g_1, u)$ is finite.
Consider a point $$p = \left((g_1, u),[0:\xx]\right)\in \pi^{-1} (g_1, u)$$ and let $V$ be an open neighborhood of $p$ in $\Gamma$ (with the usual topology) containing no other points of $\pi^{-1} (g_1, u)$ in its closure.  
Notice that $$\dim V =\dim \Gamma = \dim\left(\CC [u, \xx]_d \times \CC [u, \xx]_1\right) $$ and $U = \pi (V) \setminus \Sigma $ is a nonempty open subset of $\CC [u, \xx]_d \times \CC [u, \xx]_1$ with $\pi(p)$ in the interior of $\overline U$.
The map $\pi $ restricted to  $\pi^{-1} (U) \cap V$ is a biholomorphism onto $U.$
Since $\pi(p)=(g_1, u)$ is in the interior of $\overline{U},$ the segment $ (1-t) (g_0, \ell ) + t (g_1, u)$ intersects $U$ for values of $t$ arbitrarily close to $1.$
Let $(g_t, \ell_t) \to (g_1, u)$ along the set of points where this segment intersects $U.$
Lifting to $\pi^{-1} (U) \cap V,$ we obtain $\left((g_t, \ell_t),  [\vartilde u_t : \vartilde \xx_t] \right)
\to \left( (g_1, u) , [\vartilde u : \vartilde \xx ]\right) \in \pi^{-1} (g_1, u).$
Since $[\vartilde u:\vartilde \xx ]$ is an endpoint of $\vartilde H_t,$ we must have $\vartilde u = 0.$
Moreover, since $\pi^{-1} (g_1, u) \cap \overline V = \{p\}$, we must have $[\vartilde u : \vartilde \xx ]  = [ 0: \xx ].$
\end{proof}

\begin{ex} 
\label{ex:parabolas}
\Cref{fig:u-illustration} illustrates the homotopy $\vartilde H_t$ \eqref{eq:Htwave-Pn} for a simple case: intersecting two parabolas in the plane.
We intersect $X=\VV(F)\subset \PP^2,$ where 
\[
F (\xx) = \colorf{x_1^2 - x_0 x_2 - 2x_0^2},
\]
with $\VV (g_1) \subset \PP^2,$ where
\[
g_1 (\xx) = \colorg{2 x_0^2 + x_1 x_0 - x_2^2}.
\]
Choosing $\ell_0 (\xx) = x_2$ and $g_0 (u, \xx) = x_0^2 - u^2$ (cf.~\Cref{rem:g0-Pn}), the genericity conditions required in~\Cref{prop:endpoint} are satisfied.
The start points are the four points in 
$\VV({F}) \cap \VV ({g_0}) \cap \VV (\ell_0)\subset\PP^3.$
The four endpoints of $\vartilde H_t$ in $\PP^3$ are given in homogeneous coordinates $[u: x_0:x_1:x_2]$ as follows:
$$
\begin{array}{c|c|c|c}
u& x_0 &     x_1 &      x_2 \\
\hline
0 &  1 &      -1 &       -1 \\
0 &  1 &       2 &        2 \\
0 &  1 &   -\phi & \phi - 1 \\
0 &  1 & \phi -1 &    -\phi 
\end{array}
$$
where $\phi = (1+\sqrt{5})/2.$
The coordinates $[x_0:x_1:x_2]$ give us the four points of intersection in $\VV (\colorf{F}) \cap \VV (\colorg{g_1}).$

\begin{figure}
    \centering
    \begin{tabular}{m{8.3em} m{8.3em} m{8.3em} m{8.3em}} 
         \begin{tabular}{c}
            \\
            \includegraphics[width=8.3em]{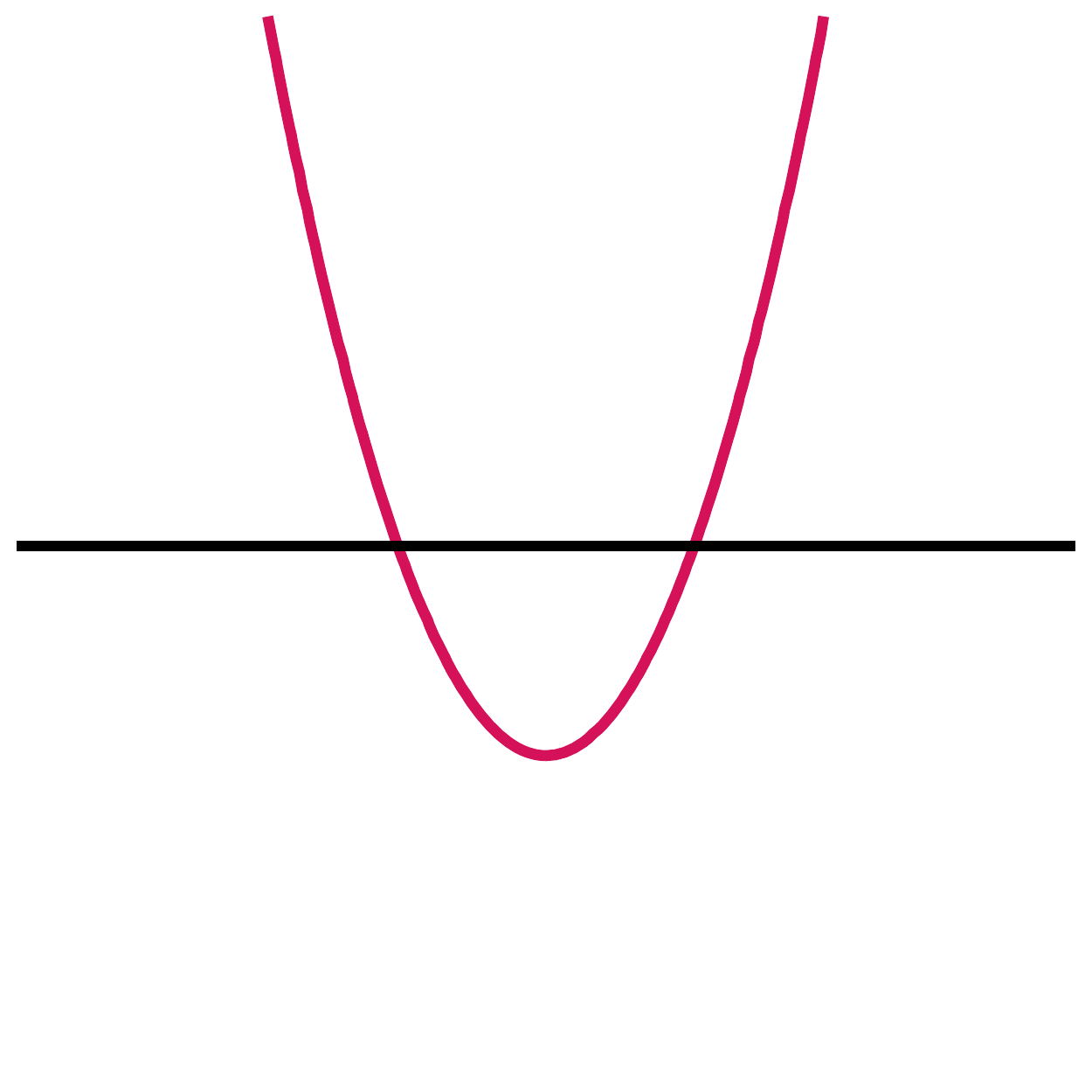}
        \end{tabular}
    &
         \begin{tabular}{c}
            $t=0$\\
            \includegraphics[width=8.3em]{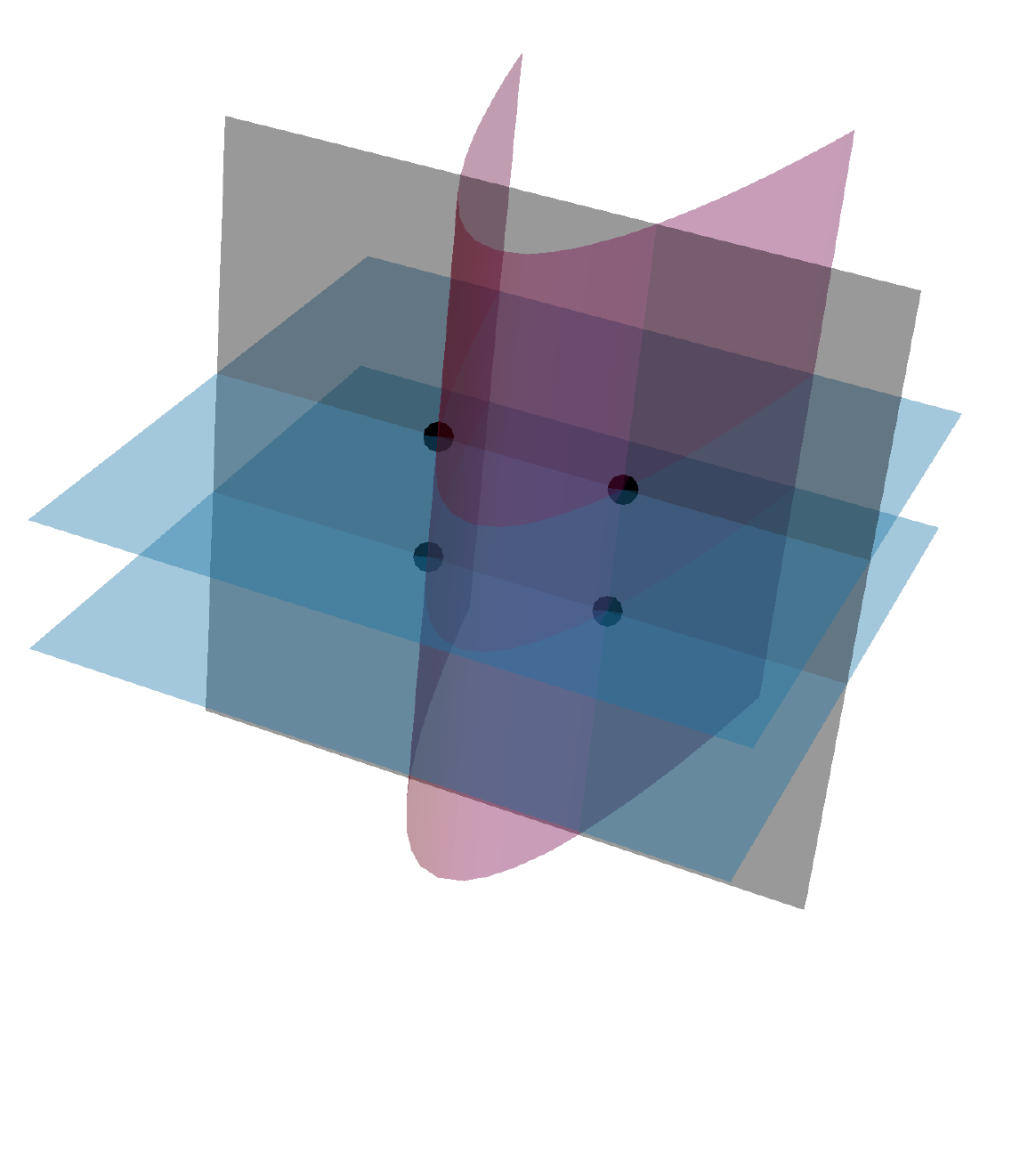} 
        \end{tabular}
    &
         \begin{tabular}{c}
            $t=.2$\\
            \includegraphics[width=8.3em]{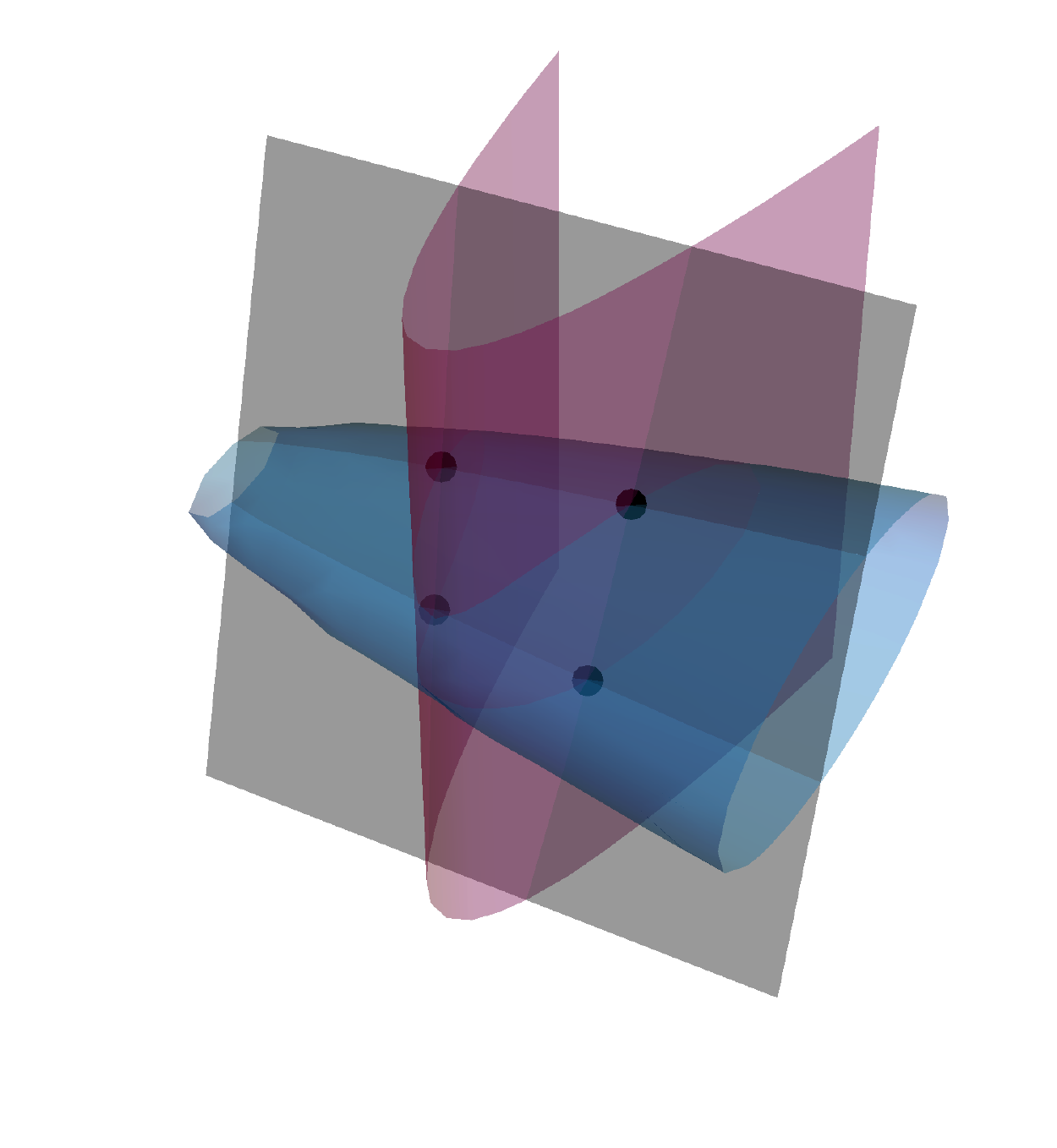}  
        \end{tabular}
    &
         \begin{tabular}{c}
            $t=.4$\\
           \includegraphics[width=8.3em]{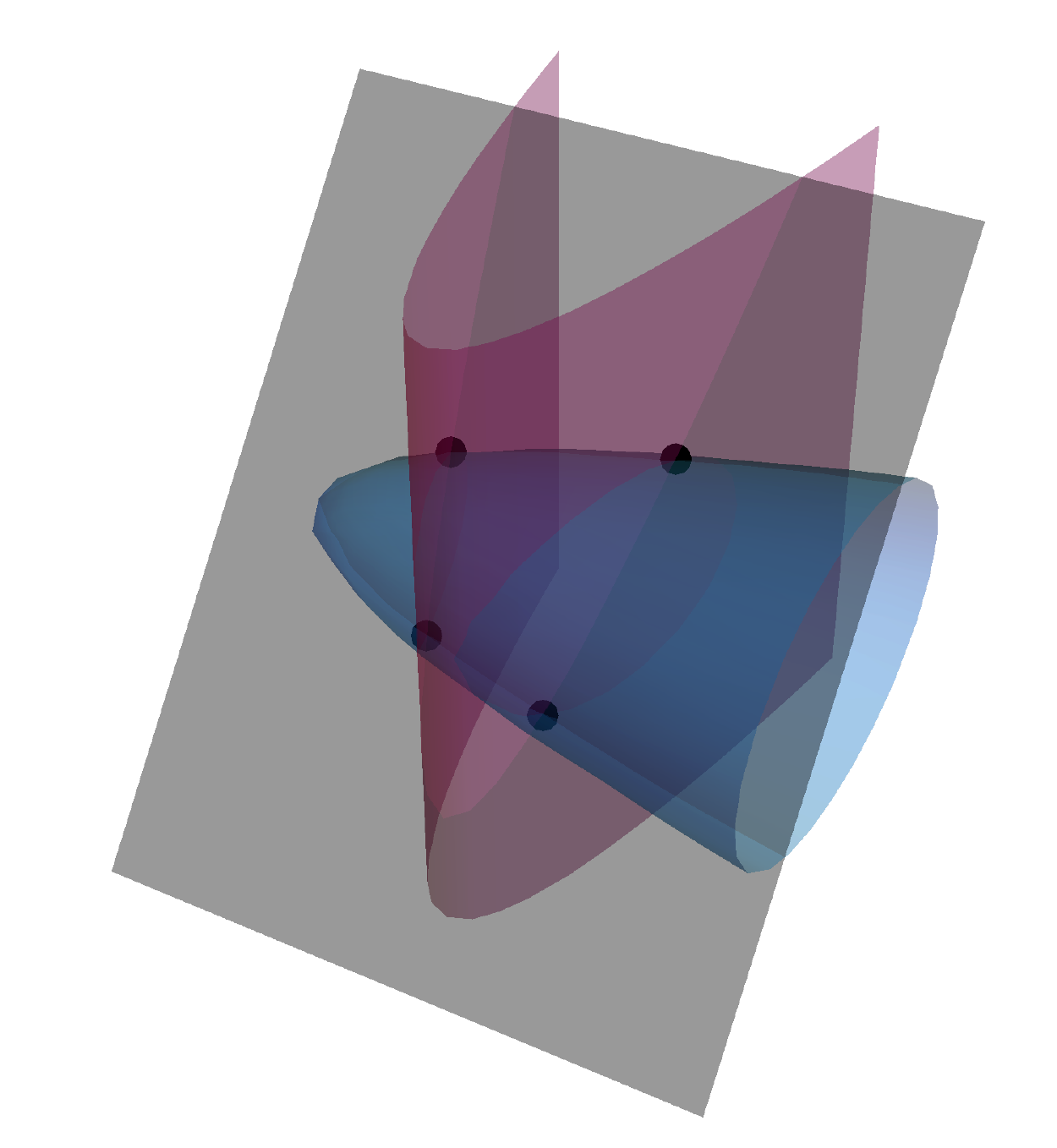} 
        \end{tabular}
    \\
        \begin{tabular}{c}
            $t=.6$\\
            \includegraphics[width=8.3em]{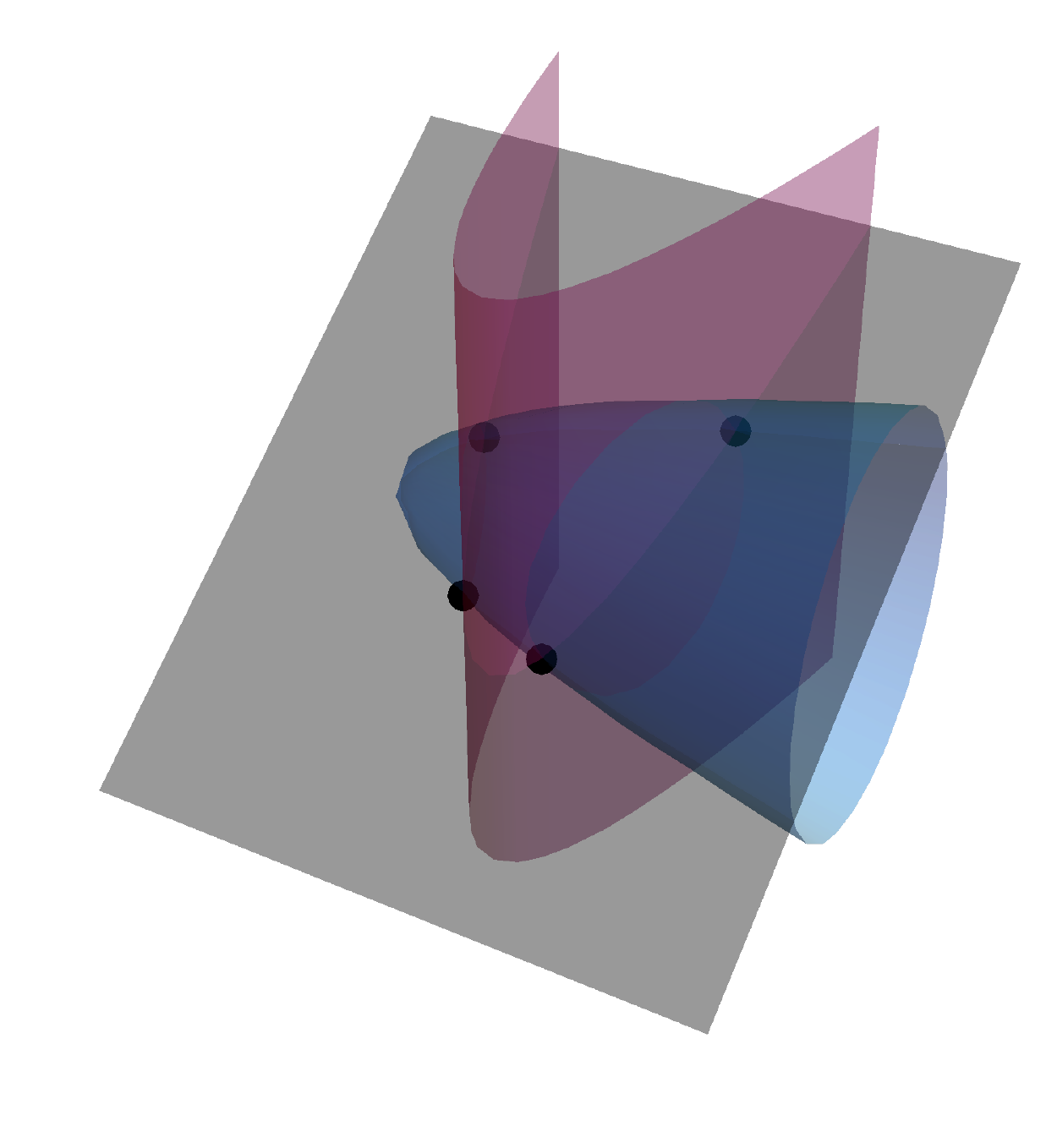}
        \end{tabular}
     &
        \begin{tabular}{c}
            $t=.8$\\
            \includegraphics[width=8.3em]{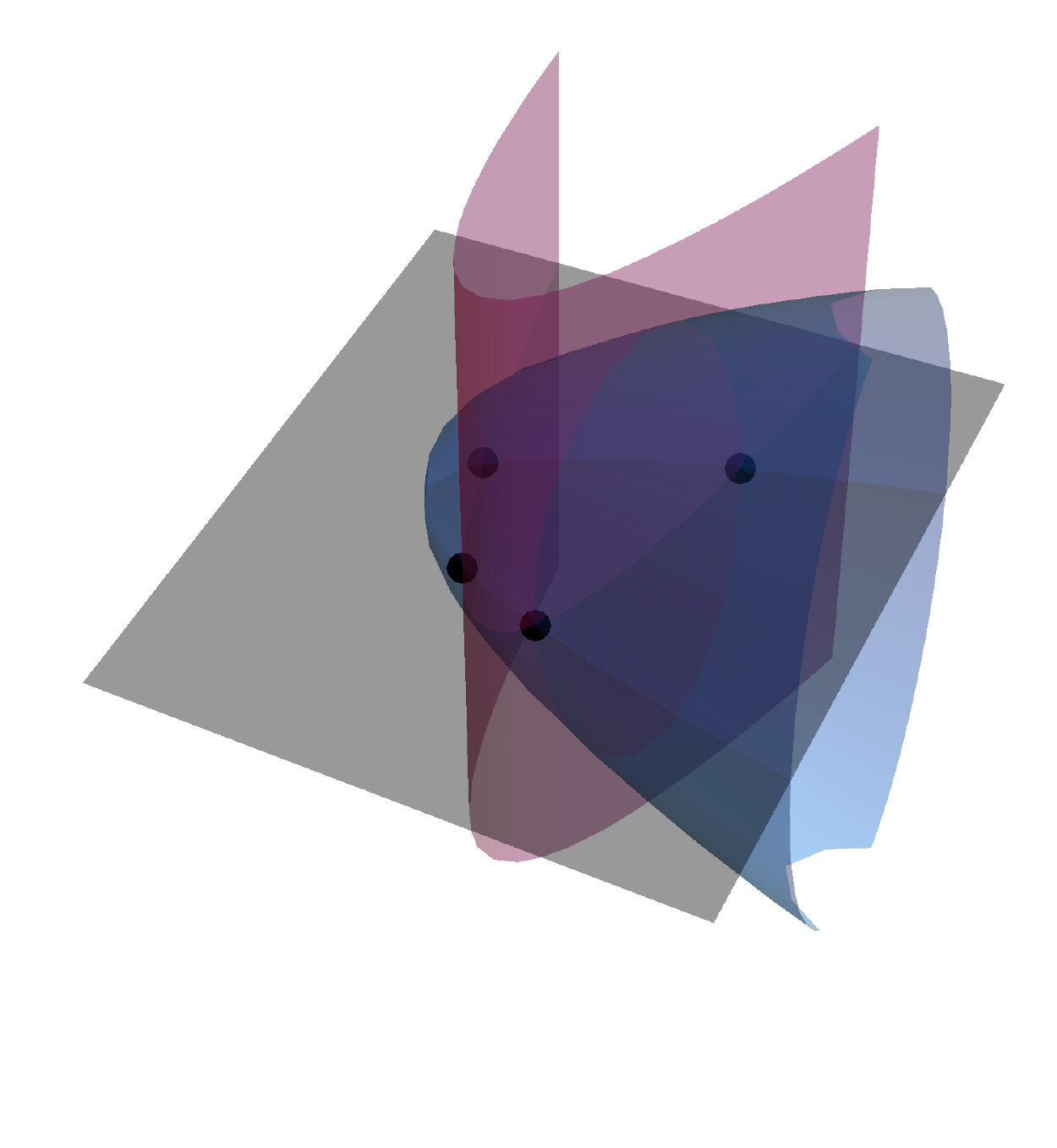}
        \end{tabular}
    &  
        \begin{tabular}{c}
            $t=1.0$\\
            \includegraphics[width=8.3em]{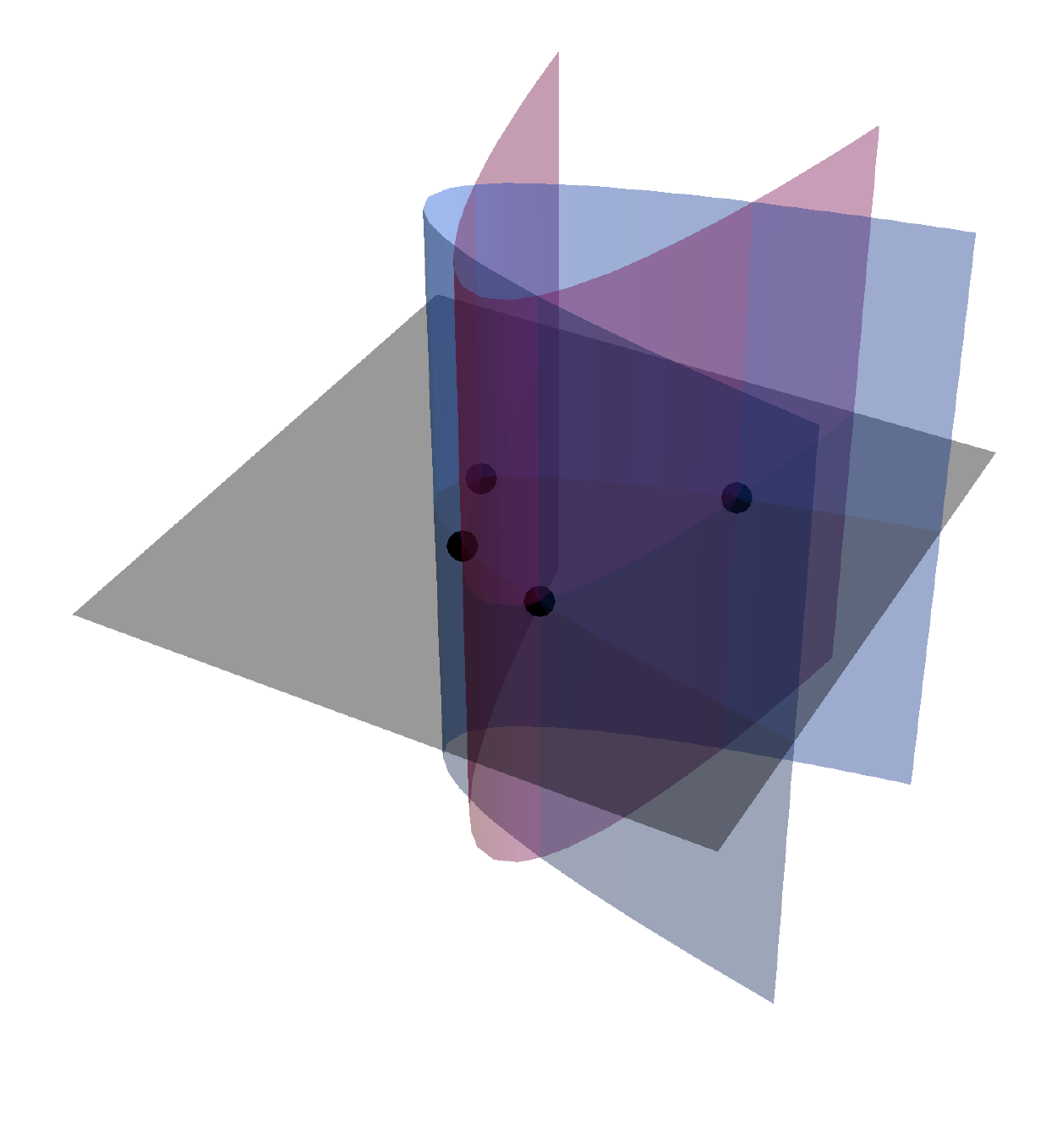} 
        \end{tabular}
    &
        \begin{tabular}{c}
            \\
            \includegraphics[width=8.3em]{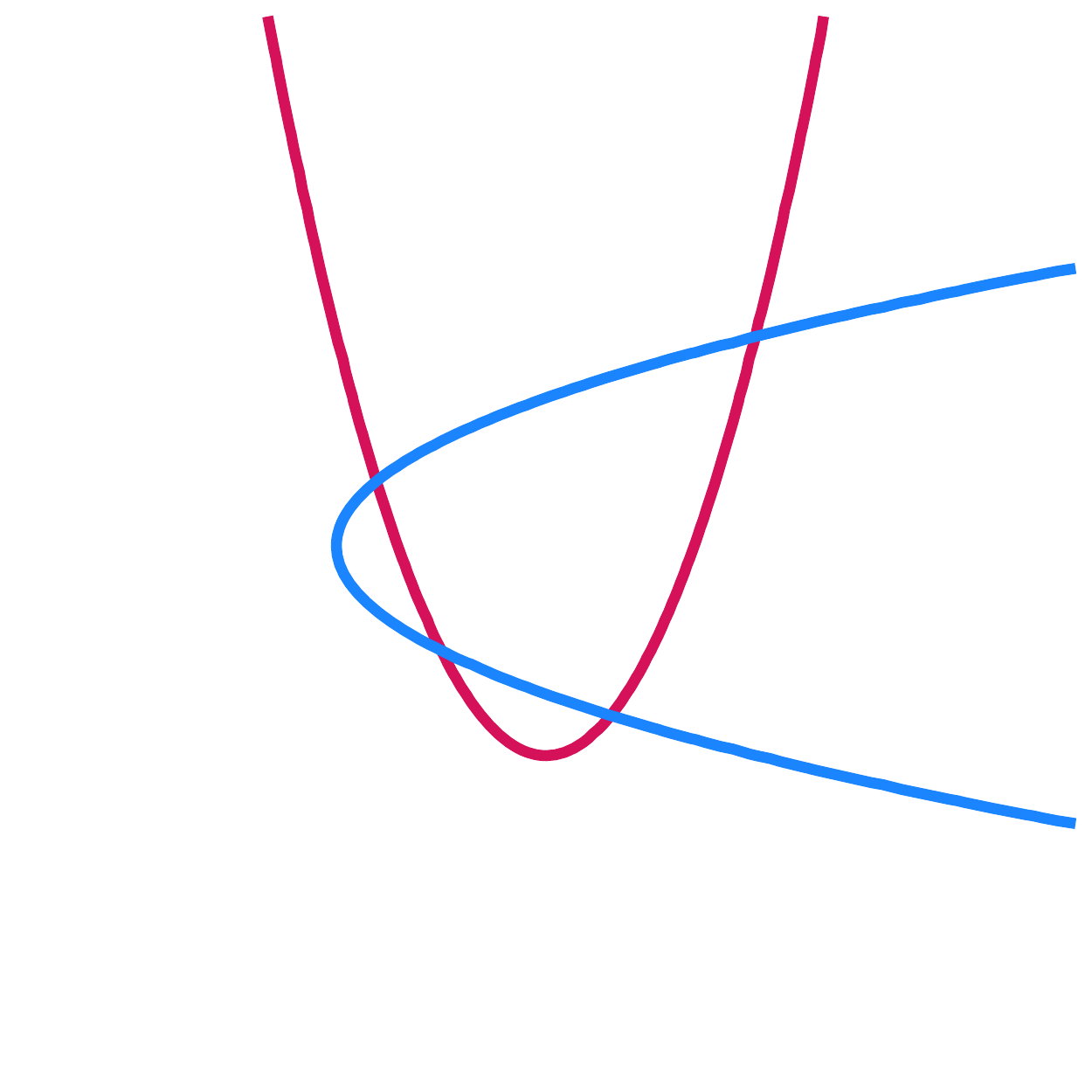}
        \end{tabular}
\end{tabular}
    
    \caption{Illustration of~\Cref{ex:parabolas} in the affine charts $x_0=1.$
    The homotopy $\vartilde H_t$ deforms a union of planes $\VV (\colorg{g_0})\subset \PP^3$ into the cone over the target hypersurface $\colorg{g_1},$ and the plane $\VV (\ell_1)$ into $\VV (u).$
    In the plane $\PP^2$, this allows us to obtain $\VV (\colorf{F}) \cap \VV (\colorg{g_1})$ from $\VV (\colorf{F}) \cap \VV (\ell_0)$.
    }
    \label{fig:u-illustration}
\end{figure}

\end{ex}

\subsection{Witness sets}\label{sec:witness-sets}
In numerical algebraic geometry, algebraic varieties are represented by witness sets. In this subsection we review witness sets and show how they relate to the start and endpoints of our homotopy \eqref{eq:Htwave-Pn}.

An equidimensional  
variety $X\subset \mathbb{P}^n$
is a finite union of irreducible varieties of the same dimension. It 
is represented by a \emph{witness set}
$w := (F,L,W)$, 
a triple consisting~of
\begin{itemize}
    \item polynomials $F$ defining $X$ such that it is a union of irreducible components of $\VV(F)$; 
    \item general 
    \footnote{We use ``general'' throughout this article to mean ``avoiding some proper Zariski closed set''. Here this exceptional set has a simple description: $\VV(L)$ needs to be outside the set of subspaces in the Grassmannian $\Gr(n-c,n)$ that do not intersect $X$ regularly. However, the randomized methods of homotopy continuation don't rely on knowing the exceptional set description: we use ``general'' without  aiming to provide such a description later on (e.g., in~\Cref{rem:g0-Pn}).} 
    linear polynomials $L$ defining a codimension $\dim(X)$ linear space, informally called a slice of $\mathbb{P}^n$;
    \item the set of points $W = X \cap \VV(L)$. 
\end{itemize}

Witness sets may be used to test if a point is in an irreducible component of a variety~\cite[Chapter 15.1]{SW05}, 
describe a wide class of varieties, including closures of images of a rational maps~\cite{MR4121336} or subvarieties of products of projective spaces~\cite{MR2733776,MR4166467} and Grassmannians~\cite{Sottile2020}.

\begin{ex}
Suppose $X\subset \mathbb{P}^n$ is a finite set of points. 
A witness set for $X$ has the form
    $w=(F,\, \emptyset, \, X)$ where each point of $X$ is an isolated point in $\VV(F)$. 
\end{ex}

When $X$ is a finite set of points (ie.~$\dim X =0$), computing the intersection of $X$ with a hypersurface is straightforward:  $X\cap V(h) = \{ x \in X : h(x)=0  \}$. 
For $X$ of arbitrary dimension, Algorithm~\ref{algo:intersection} shows how to obtain a witness set for the intersection of $X$ with a hypersurface from a witness set for $X.$
This can be done by reduction to the case of a curve and applying~\Cref{prop:endpoint}.
First, we interpret the homotopy $\vartilde H_t$ in the language of witness sets in Example~\ref{ex:witness-start}.

\begin{ex}\label{ex:witness-start}
Recall from equation \eqref{eq:Htwave-Pn} that $X$ is a curve, 
$g_0\in \CC[\xx]_d$ is general,
and the intersection $\vartilde X\cap V(g_0)$ is a  degree $\deg(X)\cdot d $ curve.
Let $w_0$ denote a witness set for $\vartilde X \cap V(g_0)$.
Then the witness points of $w_0$ are a set of start points for our homotopy \eqref{eq:Htwave-Pn}. 
Simarly, the isolated endpoints of the homotopy $\vartilde H_t$ are witness points for $\vartilde X \cap \VV (g_1),$ and projection from $u$ gives witness points for $X \cap \VV (g_1)$.

To compute $w_0$ and thereby the start points of $\vartilde H_t$, assume we are given a witness set $w=(F,\ell,W)$ for $X$.
For each point $\xx^*$ in $W$, the set 
\[
\{ [u:\xx^*] \in \PP^{n+1} : g_0(u,\xx^*)=0 \}
\]
has 
$\deg(g_0)$ points in $\vartilde X\cap \VV(g_0)$.
All together, by solving $\deg(X)$  univariate degree  $\deg(g_0)$
polynomials, we have 
\[
w_0 = (F\cup \{ g_0\}), \, \ell ,  \, \{ [u:\xx^*] \in \PP^{n+1} : g_0(u,\xx^*)=0,\, \xx^* \in W \}.
\]
We summarize this process in Algorithm~\ref{algo:start-points-for-Pn}.
\end{ex}

\begin{algorithm}[H]
\caption{$\uStartPoints(w,g_0)$}
    \label{algo:start-points-for-Pn}
    
\KwIn{
    A witness set $w=(F,\ell,W)$ for a curve $X\subset\PP^n$;\\
    a general homogeneous polynomial $g_0\in \CC[u,\xx]$ of degree $d$.
}
\KwOut{
    A witness set $w_0$ representing $\vartilde X\cap\VV(g_0) \subset\PP^{n+1}$. 
}
\smallskip
    Compute  $$S_0 \gets \{ [u:\xx]  \mid \xx\in W \text{ and } g_0(u,\xx)=0 \},$$ which requires solving a univariate polynomial equation for each $\xx \in W$.
    
    Let $
    w_0 \gets \big(F\cup\{ g_0 \},\, \{ \ell \} ,\, S_0 \big).
    $ \CommentSty{ // Here $F$ and $\ell$ (which do not depend on $u$) should be seen as elements of $\CC[u,\xx]$.}
\end{algorithm}


One important fact to recall about intersecting an irreducible projective variety $X$ with a hypersurface $\VV(h)$ is that it leads to one of two cases: 
either $X\cap \VV(h)=X$ or $X\cap \VV(h)$ is equidimensional with dimension $\dim(X)-1$. 
The former case only occurs when $X$ is contained in the hypersurface. In the latter case the dimension decreases by one. Moreover, the degree of the intersection is bounded above by $\deg(h)\cdot \deg(X)$.

More generally, if we drop the irreducibility hypothesis and only assume $X$ is equidimensional, then the intersection $X\cap \VV(h)$ consists of two equidimensional components: 
the union of all irreducible components of $X$  that are also contained $\VV(h)$, and 
$\overline{(X\setminus \VV(h))\cap \VV(h)}$.
The first equidimensional component has the same dimension as $X$ and the second  has dimension $\dim(X)-1$. 

As an immediate consequence of these facts, the results in the previous sections,  and these examples, we have Algorithm~\ref{algo:intersection}. 

\begin{algorithm}[H]
\DontPrintSemicolon 
\KwIn{
    A homogeneous polynomial $g_1\in \CC[\xx]$, and 
    a witness set $w=(F,L,W)$ for a pure $\dimX$-dimensional variety $X\subset\PP^n$. 
    \;
}
\KwOut{
    Witness sets for the equidimensional components of $X\cap\VV(g_1)$.\; 
}
\label{item:partition} Record the witness set
    \[w^{(\dimX)}\leftarrow (F, L, W \cap  V(g_1) )\] 
    for  the union of irreducible components of $X$ contained in $V(g_1)$.
    \;
\eIf{$r\geq 1$}{
    Let $\ell$ denote a linear polynomial in $L$.\;
    For the curve $Z:= \overline{(X \setminus \VV (g_1))}\cap V(L\setminus \{\ell \})$, form the witness set
        \[
        w_Z \leftarrow (F\cup L\setminus \{\ell\}, \,
        \{\ell\}, \,
         W \setminus  \VV(g_1)).
        \]
    \vspace{-12pt}   
    \label{item:update-to-curve}
    \;
Choose $g_0\in \CC[u,\xx]_{\deg g_1}$ 
    (e.g., as in \Cref{rem:g0-Pn}) 
    to form the homotopy
    \begin{equation}
    \label{eq:Htwave-Pn-Z}
    \vartilde H_t : (g_0,\ell)
    \leadsto
    (g_1,u)\quad \text{ on }\vartilde Z.
    \end{equation}
    \vspace{-12pt}   
    \label{item:make-homotopy}
   \;
Follow $|W \setminus  \VV(g_1)|\cdot\deg(g_1)$ homotopy paths starting (at $t=0$) from the witness set for $\vartilde Z\cap \VV(g_0)$ returned by $\uStartPoints(w_Z,g_0)$ to obtain endpoints $E = \vartilde Z \cap \VV(g_1,u)$ at $t=1$.
    \label{item:use-homotopy}
   \;
Record the witness set 
        \[
        w^{(\dimX-1)} \leftarrow (F\cup \{ g_1\}, \, L\setminus \{\ell\} ,\,
        \pi_\xx(E)),\]
        where $\pi_\xx$ is the projection that drops the $u$-coordinate, which represents  $\overline{ (X\setminus \VV(g_1) )} \cap \VV(g_1)$. 
        \label{item:update-to-intersection}
        \;
\Return{ $\{w^{(\dimX)},w^{(\dimX-1)}\}$}\;
    }{
      \Return{ $\{w^{(\dimX)}\}$}\;
    }

\caption{$\IntersectionWithHypersurface(w,g_1)$}
\label{algo:intersection}
\end{algorithm}

\begin{remark}\label{rem:g0-Pn} 
When implementing Algorithm~\ref{algo:intersection} and subsequent algorithms, the polynomial $g_0$ should be chosen in a sufficiently random fashion. At the same time it is desirable to have a low evaluation cost for both the values of $g_0$ and the roots of $g_0(u,\xx)=0$ when $\xx$ is fixed.

One good candidate is $g_0=\gamma(u^d-\ell_0^d)$ where $d=\deg(g_0)$, the linear form $\ell_0\in\CC[\xx]_1$ doesn't vanish on the start points, and $\gamma\in\CC$ is generic\footnote{The constant $\gamma$ needs to avoid a \emph{real} hypersurface (in fact, a union lines through the origin) in $\RR^2\cong\CC$ containing ``bad'' choices: a choice of $\gamma$ is ``bad'' if the homotopy paths of \eqref{eq:Htwave-Pn} cross.}.
This is akin to the so-called \emph{$\gamma$-trick}~\cite[pp.~94--95, 120--122]{SW05}. 
Often, for practical purposes, $\gamma$ is chosen randomly on the unit circle. In this case, only finitely many points on the circle are exceptional.
\end{remark}

One other important detail for a reader who intends to implement Algorithm~\ref{algo:intersection} is in the following remark. 

\begin{remark}\label{rem:eliminate-u}
A typical implementation of a homotopy tracking algorithm would introduce an affine chart on $\PP^{n+1}$ by imposing an additional affine linear equation in $u$ and $\xx$.
One can eliminate $u$ in two ways: (a) using the chart equation or (b) using the equation $(1-t)\ell + tu=0$ from homotopy \eqref{eq:Htwave-Pn}. This results in going back to the original ambient dimension. Note that solving (b) for $u$ is not possible when $t=0$. 
\end{remark}

\subsection{Equation-by-equation cascade} 

Algorithm~\ref{algo:cascade} uses the methods in Section~\ref{ss:homotopy} to compute the witness set(s) for the intersection of a variety with a hypersurface.

\begin{algorithm}[H]
\DontPrintSemicolon 
\KwIn{
    $F=(f_1,\dots,f_c),\ f_i \in \CC[\xx]$
}
\KwOut{
    A collection $C$ of witness sets describing equidimensional pieces of~$\VV(F)$.
}

Let $w = (\emptyset, \{\ell_1,\dots,\ell_n\}, \VV(\ell_1,\dots,\ell_n))$, a witness set of a point, describing~$\PP^n$.\;
Initialize $C \leftarrow \{w\}$.
\;

\For{$i$ from $1$ to $c$}{
$C' \leftarrow \emptyset$
\;

\For{each $w \in C$}{
$C' \leftarrow C' \cup \IntersectionWithHypersurface(w,f_i)$\;
}

$C \leftarrow \text{\EliminateRedundantComponents}(C')$, a routine that erases components contained in other components.  
}
\caption{Equation-by-equation cascade}
\label{algo:cascade}
\end{algorithm}

The correctness of the equation-by-equation cascade outlined Algorithm~\ref{algo:cascade} hinges mainly on Algorithm~\ref{algo:intersection}. 
However, an implementer of the cascade must pay significant attention to details with a view to practical efficiency. 
For instance, ``pruning'' of the current collection of witness sets done via $\EliminateRedundantComponents$ after each step could be replaced by a potentially more efficient {\em en route} bookkeeping procedure; see \cite{HSW:regenerative-cascade}.

\begin{remark}
Our description of homotopies is geometric: we consider homotopies on a variety $X$ (or the cone $\vartilde X$.) However, the algebraic representation of $X$ as a component of $\VV(F)$ may not be \emph{reduced}, i.e. $F$ may be singular at some points of $X$. 
A common way to address this issue in an implementation is \emph{deflation}, which is a technique of replacing $F$ with a system of polynomials (potentially in more variables) that are regular at a generic point: see \cite[\S 13.3.2]{SW05} and \cite{LVZ-deflation}.
		
\end{remark}

\section{\TeXheadline{$u$}-generation in products of projective spaces}
\label{sec:multiprojective}
In this section, we generalize the homotopy $\vartilde H_t$ to work in a product of projective spaces $\PP^{n_1}\times\cdots \times \PP^{n_k}$. 
To ease exposition, we restrict the discussion to the case where $k=2,$ a product of two projective spaces. All results extend to an arbitrary number of factors.
In~\Cref{subsec:MLE}, we study an example with $k=3$ factors in detail.

\subsection{Setting up the homotopy}
We define the \demph{double cone} of a biprojective variety $X \subset \PPtwo$ as 
\[
\vartilde X = \overline{\{ \left([u:\xx],\, [v:\yy] \right) \in \PP^{m+1} \times \PP^{n+1} \mid (\xx , \yy) \in X \}}.
\]

In the example below, we consider simple but illustrative case where $X$ is a single point.
This provides some visual intuition for the homotopy paths of $\tilde{H_t}$.

\begin{ex}[Double cone over a point]
For a point
$P=(\xx^\star,\yy^\star)\in \PPtwo$, the double cone $\vartilde P $ is the $2$-dimensional biprojective linear space spanned by the points 
\[
([1:\xx^\star],\, [1:\yy^\star]),\,
([0:\xx^\star],\, [1:\yy^\star]), \text{ and }
([1:\xx^\star],\, [0:\yy^\star]).
\]
Consider the set of $d$ points in the intersection $\vartilde P \cap \VV(\colorg{g},L_\epsilon)$, where the hypersurface $\VV (\colorg{g})$ is given by general $\colorg{g}\in \CC[\{u,\xx\},\{v,\yy\}]_{(d,e)}$ and the family of hyperplanes defined by $L_\epsilon:=\epsilon v-(1-\epsilon) \ell$ where $\ell\in \CC[\yy]_1$ is general. 
This set of points may be recovered in two easy steps:
\begin{itemize}
\item[1.] Solve $L_\epsilon (\yy^\star, v) = 0$ to obtain $\colorf{v^\star} = \epsilon \ell(\yy^\star)/(1-\epsilon)$.
\item[2.] Compute the $d$ roots $\colorf{u_1^\star}, \ldots , \colorf{u_d^\star}$ of the univariate polynomial
$\colorg{g}(u,\xx^\star, \colorf{v^\star},\yy^\star)$.
\end{itemize}

Projecting a general affine chart of $\vartilde P$ to the affine $uv$-plane $\CC^2_{uv}$ is a surjection. 
The intersection $\vartilde P\cap \VV(g,L_\epsilon)$ can be visualized as the intersection of a line $\VV(L_\epsilon(v,\yy^\star))$ with the plane curve defined by  $\colorg{g}(u,\xx^\star,v,\yy^\star) =0.$ 
Figure~\ref{fig:uv-intersection} illustrates a case where $(d,e)=(2,1).$ 
In the limit as $\epsilon\to0$, $\colorf{v^\star}$ tends toward infinity, while the roots $\colorf{u_1^\star}, \ldots , \colorf{u_d^\star}$ tend towards the $d$ vertical asymptotes of the curve given by $\colorg{g}$. 
If we write $\colorg{g}(u,\xx^\star,v,\yy^\star) = f(u) \, v^d + \text{additional terms},$ these asymptotes cross the $u$-axis at the points where $u$ is a root of $f.$
\begin{figure}[htb!]
    \centering
\begin{tabular}{ccc}
\includegraphics[width=12em]{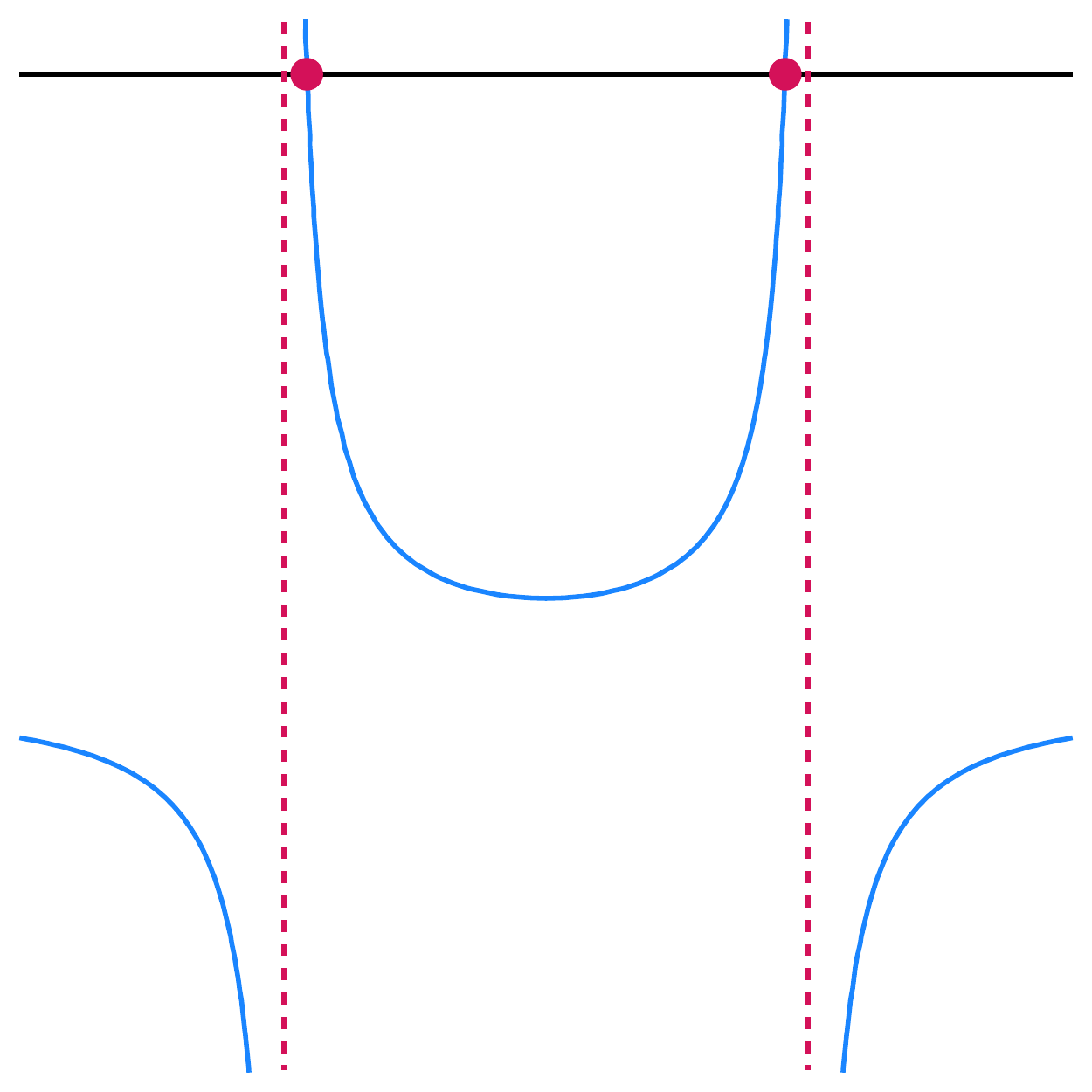} & &
\includegraphics[width=12em]{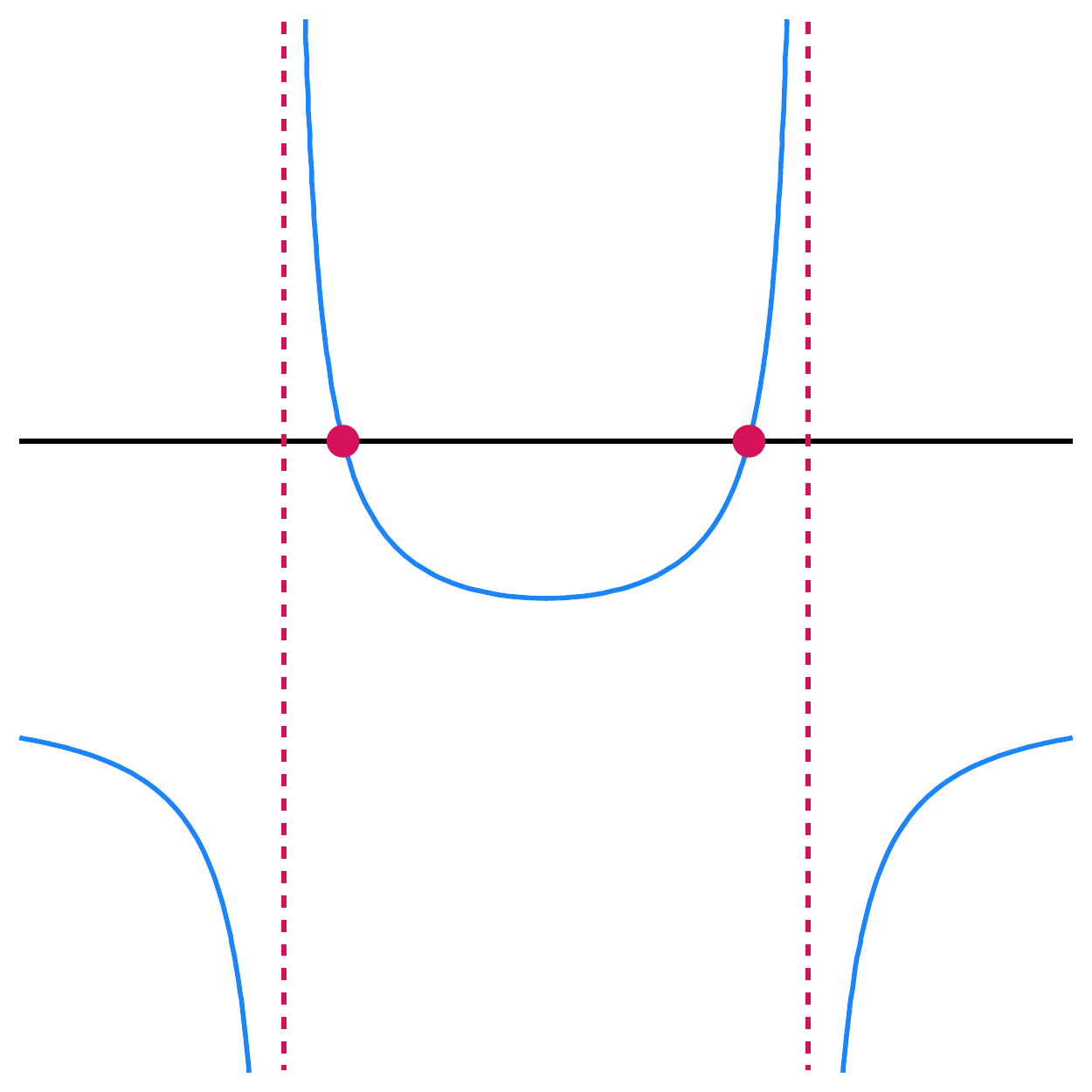}\\
$\epsilon \approx 0$ & & $\epsilon \gg 0$
\end{tabular}
    \caption{
    Interesecting the double cone over a point $\tilde P$ with $\VV(\colorg{g},L_\epsilon)$, producing points $(\colorf{u_1^\star}, \colorf{v^\star}), (\colorf{u_2^\star}, \colorf{v^\star})$ in the affine plane plane $\CC_{uv}^2.$
    }
    \label{fig:uv-intersection}
\end{figure}
\end{ex}


Now we consider the double cone over a curve $X\subset\PPtwo$ and state how they share some of the same degree information. 
Following the conventions in~\cite{MR4166467}, 
the bidegree of a curve $X$ in $\PPtwo$
with coordinate ring $\CC[\xx,\yy]$ 
is  
\[
( \deg_{\xx}(X),\deg_\yy(X) ):= ( \deg(X\cap \VV(\ell_{\xx})), 
    \deg(X\cap \VV(\ell_{\yy}) )
    ) 
\]
where $\ell_\xx\in \CC[\xx]_1$, $\ell_\yy\in \CC[\yy]_1$ are general. 
Let $\ell_{u,\xx}\in \CC[u,\xx]_1$ and $\ell_{v,\yy}\in \CC[v,\yy]_1$
be general.  
Then the bidegrees of the curves $X\subset \PPtwo$ and  $\vartilde X \cap \VV(\ell_{u,\xx},\ell_{v,\yy})\subset \PP^{m+1}\times \PP^{n+1}$ coincide. 

\begin{remark} A polynomial $h\in \CC[\xx,\yy]_{(d,e)}$ is defined to have bidegree $(d,e)$. 
When $(m,n)=(1,1)$, the polynomial $h$ defines a curve $\VV(h)\subset\PP^1\times\PP^1$.
The bidegree of $\VV(h)$ and bidegree of the polynomial $h$ are related by a transposition:
$(\deg_{\xx}(\VV(h)), \deg_{\yy}(\VV(h)) ) =(e,d).$
\end{remark}

The following proposition is a needed analogue to \Cref{prop:endpoint}.
Analogously to the homotopy in~\Cref{eq:Htwave-Pn}, for given $g_1 \in \CC[\xx , \yy]_{d,e}$ we define a homotopy
\begin{equation}\label{eq:two-hom-homotopy}
\vartilde H_t:=(g_0, \ell_\xx, \ell_\yy )
\leadsto
(g_1, u, v) \quad \text{ on }\vartilde X
\end{equation}
where $g_0 \in \CC[\{u,\xx\}, \{v,\yy\}]_{d,e}$.

\begin{proposition}\label{prop:two-factor-homotopy}
For generic $\ell_\xx \in \CC[\xx]_1$, $\ell_\yy \in \CC[\yy]_1$, $g_0 \in \CC[\{u,\xx\} , \{v,\yy\}]_{d,e},$ the cardinality of the set of points on $\vartilde X$ satisfying $\vartilde H_t$  equals, for all $t\in (0,1),$  
\[
d\cdot \deg_\xx(X) +e\cdot \deg_\yy(X).
\]
The endpoints of $\vartilde H_t$ lie in the set
\[
\left\{ ([0:\xx], [0:\yy]) \in \PP^{m+1} \times \PP^{n+1} \mid (\xx , \yy ) \in X \cap \VV (g_1).
\right\}
\]
In case this set is finite, every point is reached.
\end{proposition}

\begin{proof}
The homotopy $\vartilde H_t$ 
is given by 
\begin{equation}\label{eq:bi-homotopy-equations}
g_t:=(1-t)g_0+tg_1 = 0, \quad
(1-t)\ell_\xx+tu = 0,\quad
(1-t)\ell_\yy+tv = 0
\end{equation}

For $t\in (0,1)$, the intersection
\begin{equation}\label{eq:zt}
Z_t:=\vartilde X\cap 
\VV ( 
(1-t)\ell_\xx+tu,
(1-t)\ell_\yy+tv ).
\end{equation}
is a curve in $\PP^{m+1}\times \PP^{n+1}$ with the same bidegree as $X \subset \PP^{m}\times \PP^{n}$, i.e.,  
\[
\deg_\xx(X)=\deg_{u,\xx}(Z_t)
\text{ and }
\deg_\yy(X)=\deg_{v,\yy}(Z_t).
\]
since the last two equations give $u$ and $v$ as a linear functions of $\xx$ and $\yy$, respectively, when $t\in (0,1)$.

Since $g_0$ is general, so is $g_t$. Hence,
for $t\in (0,1)$ we get that
$Z_t\cap V(g_t)$ has 
$d\cdot \deg_{u,\xx}(Z_t)+e\cdot \deg_{v,\yy}(Z_t)$
points of intersection. 
This proves the first part of the proposition. 

The proof of the second part is similar to the proof of the second part of~\Cref{prop:endpoint} concluding that
a point $[\xx_1:\xx_2]\in X\cap V(g_1)$
corresponds to an endpoint $([0:\xx_1],[0:\xx_2])$ of the homotopy.
\end{proof}

\begin{remark}\label{rem:multiprojective-points-at-t0}
If $X$ is a general bidegree $(d_\xx,d_\yy)$ curve in $\PPtwo$, 
then 
$Z_0$, as defined in \eqref{eq:zt}, 
is a union of two linear spaces with multiplicity $d$ and $e$ respectively. Specifically, 
\begin{eqnarray*}
Z_0&=&\{
([u:\xx], [1:\mathbf{0}]) \in \vartilde X
: \ell_\xx(\xx)=0
\} \ \cup\\
&& 
\{
([1:\mathbf{0}]  , [v:\yy] ) \in \vartilde X
: \ell_\yy(\yy)=0
\}.
\end{eqnarray*}
On the other hand, the homotopy $\vartilde H_t$ of~\eqref{eq:two-hom-homotopy} converges to 
\begin{equation}\label{eq:bi-start-points}
\begin{array}{rcl}
\VV(H_0) \cap \vartilde X &=& \{ ([u:\xx],[1:\bfzero]) \mid \\
&& \xx \in \pi_\xx(X\cap\VV(\ell_\xx)), \text{ and } u \text{ satisfies } g_0(u,\xx,1,\bfzero) = 0 \}\\
&\cup& \{ ([1:\bfzero],[v:\yy]) \mid \\
&& \yy \in \pi_\yy(X\cap\VV(\ell_\yy)), \text{ and } v \text{ satisfies } g_0(1,\bfzero,v,\yy) = 0 \}
\end{array}
\end{equation}
when $t\to 0$.
Thus multiple paths may tend to the same point as $t\to 0$. 

\end{remark}

\subsection{Points for the homotopy at \TeXheadline{$t=\varepsilon$}}
In view of \Cref{rem:multiprojective-points-at-t0} (see also \Cref{rem:double-cone-artifacts}), we need to step away from $t=0$ and find a way to produce a set of start points for the homotopy for $t=\varepsilon>0$ for a small $\varepsilon$.



 Fix an affine chart on  $\PP^{m+1}\times\PP^{n+1}$; e.g., take $x_0=y_0=1$. Without loss of generality we may assume $x_0$ and $y_0$ don't vanish on any of the witness points.  
 Consider a continuation path that tends to the point of the form $([u^*:\xx^*] , [1:\mathbf{0}])$ as $t\to 0$ (a point that is not in the fixed chart) in a small punctured neighborhood of $t=0$. 

 View the coordinates  $(u(t),\xx(t),v(t),\yy(t))$ of this path as (convergent in the punctured neighborhood) Puiseux series where $u(t),\xx(t),\yy(t)$ have a (nonzero) constant leading term while $v(t)$ diverges as $t\to 0$. Note that the constant terms of $\xx(t),\yy(t)$ are given by a witness point $[\xx^*:\yy^*] \in X\cap \VV(\ell_\xx)$ (with $x_0=y_0=1$).
 
 First, we note that since the path converges in $\PP^{m+1}\times\PP^{n+1}$  
 to a point $([u^*:\xx^*] : [1:\mathbf{0}])$ the leading term of $u(t)$ is a constant term. Plug $u(t)$ in the second equation in \eqref{eq:bi-homotopy-equations}, 
 \[
 (1-t)\ell_\xx+tu = 0
 \]
 and consider it on $X$ (not $\vartilde X$).
 It defines a continuous family of (transverse) $\xx$-slices of $X$ which limit at $\VV(\ell_\xx)$ showing that $\xx(t)$ and $\yy(t)$ tend to $\xx^*\neq 0$ and $\yy^*\neq 0$ as $t\to 0$.  

 Recall that the end of the path has $[1:\mathbf{0}]$ for $[v(t):\yy(t)]$, therefore, $v(t)$ has to diverge as $t\to 0$. 
Since $(1-t)\ell_\yy+tv = 0,$ we have 
$$
v(t) = (1-t^{-1})\ell_\yy(\yy(t)).
$$
Thus the leading term of $v(t)$ is $-\ell_\yy(\yy^*)t^{-1}$.

In summary, for a small $\varepsilon$, one may take  $(u(\varepsilon),\xx(\varepsilon),v(\varepsilon),\yy(\varepsilon)) \approx (u^*,x^*,v^*,y^*)$ where $v^* = -\varepsilon^{-1}\ell_\yy(y^*)$.
Another asymptotically equivalent choice, which would satisfy the last equation, is $v^* = (1-\varepsilon^{-1})\ell_\yy(y^*)$.

An analogous argument holds if one reverses the roles of $(u,\xx)$ and $(v,\yy)$.
Knowing the asymptotics of continuation paths as $t\to 0$ results in  Algorithm~\ref{algo:bi-hom-start-points} written in a chart-free form. 

\begin{algorithm}[h!]
\DontPrintSemicolon 
\caption{$\uMultiprojectiveStartPoints(w_\xx,w_\yy,g_0,\varepsilon)$}
\label{algo:bi-hom-start-points}
\KwIn{
    A biprojective curve $X\subset\PPtwo$
    represented by two witness sets:\\
    $w_\xx=(F,\ell_\xx,W_\xx)$ where $\ell_\xx\in \CC[\xx]_1$ and
    $w_\yy=(F,\ell_\yy,W_\yy)$ where $\ell_\yy\in \CC[\yy]_1$;\\
    a general bihomogeneous polynomial $g_0\in \CC[\{ u,\xx\},\{v,\yy \}]_{(d,e)}$;\\ 
    a number $\varepsilon>0$.
    \;
}
\KwOut{
    A set of points that approximates $\VV(\vartilde H_\varepsilon)$ for the homotopy \eqref{eq:two-hom-homotopy}. \;
}
Initialize $S\leftarrow \emptyset$.
    \;
Let $g_\xx\in \CC[u,\xx]$ be defined as $g_\xx(u,\xx)=g_0(u,\xx,1,\mathbf{0})$.\; 
\For{ $P\in W_\xx$}{
    Update $S\leftarrow S \, \cup  \,\vartilde P \cap \VV( g_\xx, 
    (1-\varepsilon)\ell_\yy+\varepsilon v)$
    }
Let $g_\yy\in \CC[v,\yy]$ be defined as $g_\yy(v,\yy)=g_0(1,\mathbf{0},u,\xx)$.\; 
\For{ $Q \in W_\yy$}{
    Update $S \leftarrow S \, \cup \, \vartilde Q \cap \VV( g_\yy, 
    (1-\varepsilon)\ell_\xx+\varepsilon u)$
    }
\Return{ S
    }
    \;
\end{algorithm}

\begin{remark}\label{rem:double-cone-artifacts}
Note that, in Algorithm~\ref{algo:bi-hom-start-points}, starting at $t=\varepsilon>0$ may be necessary not only due to several paths converging to the same point at $t=0$
as pointed out in \Cref{rem:multiprojective-points-at-t0},
but also for the following reason that is bound to play a role in a practical implementation. 

Let $F\subset\CC[\xx,\yy]$ be polynomials that cut out the curve $X\subset\PP^m\times\PP^n$, i.e., $X$ is a component of $\VV(F)$. The double cone $\vartilde X$ is a component of $\VV(\vartilde F)$ where $\vartilde F\subset\CC[u,\xx,v,\yy]$ are polynomials $F$ recast in a new ring. While this description of $\vartilde X$ is practical and retains properties essential to our approach (as explained in \Cref{prop:two-factor-homotopy}), the variety $\VV(\vartilde F)$ may possess extraneous components that contain points of \eqref{eq:bi-start-points} (even if they are regular points for $\vartilde H_t$). 

\end{remark}

\begin{remark}
When $g_0=g_1$, the homotopy \eqref{eq:two-hom-homotopy} used in Algorithm~\ref{algo:bi-hom-start-points} can be visualized in terms of intersecting double cones over a point with a hyperplane and hypersurface as in Figure~\ref{fig:uv-intersection}.
Namely, we take the double cone over each point $P$ in $W_\xx$ (analogously for $Q$ in $W_\yy$) and intersect it with $\VV\left(g_0, (1-\epsilon )\ell_\yy +\epsilon v\right)$ 
as $\epsilon$ varies between $0$ and $1$.
\end{remark}

\begin{remark} 
\label{rem:g0-multi}
The homotopy $\vartilde H_t$ depends on the choice of sufficiently generic $g_0.$
For implementation purposes one may pick (use $\gamma$-type tricks if necessary)
\begin{enumerate}
    \item $g_0 = (u^d-\ell_\xx ^d) (v^e-\ell_\yy^e)$, 
    where $\ell_\xx \in \CC[\xx]_1$,  $\ell_\yy\in \CC[\yy]_1$ are general linear polynomials (similar to \Cref{rem:g0-Pn}), or
    
    \item $g_0 = \prod_{i=1}^d (u-\ell_{\xx,i})\prod_{j=1}^e (v-\ell_{\yy,j})$
     where $\ell_{\xx,i} \in \CC[\xx]_1$,  $\ell_{\yy,j}\in \CC[\yy]_1$ are general linear polynomials.  
\end{enumerate}

\end{remark}

As mentioned in the beginning of this section, the results generalize straightforwardly to a curve $X$ in a product of $k$ projective spaces in any dimensions. 

\begin{remark}
In $u$-generation for multihomogeneous systems one may eliminate the additional variables ($u$ and $v$ in the case of two projective factors), as in   
\Cref{rem:eliminate-u} when using generic affine charts in the projective factors. 

However, there is a caveat: the conditioning of the resulting homotopy paths becomes much worse at $t=\varepsilon$, for small $\varepsilon$, in both cases (a) and (b) of \Cref{rem:eliminate-u}. In a practical implementation, one should eliminate additional variables only when the path tracking routine moves sufficiently far from $t=0$.  
\end{remark}

\subsection{Beyond curves}
In the projective case we show in Algorithm~\ref{algo:intersection} how to apply $u$-generation for intersecting a $d$-dimensional projective variety with hypersurface.
Some additional work is required to extend these observations to the multiprojective setting. 
The variety $X$ must be now be represented by a \emph{multiprojective witness collection}~\cite[Definition 1.2]{MR4121336}.
For example, the witness collection of an irreducible $3$-fold $X=\VV(F)$ in $\mathbb{P}^3\times\mathbb{P}^1$ has a collection of witness point sets:
\begin{align*}
w_{3,0}:=\VV(F)\cap \VV(\ell_{\xx,1},\ell_{\xx,2},\ell_{\xx,3}),\quad 
w_{2,1}:=\VV(F)\cap \VV(\ell_{\xx,1}, \ell_{\xx,2},\ell_{\yy , 1}),\quad \\
w_{1,2}:=\VV(F)\cap \VV(\ell_{\xx,1},\ell_{\yy,1},\ell_{\yy , 2}),\quad
w_{0,3}:=\VV(F)\cap \VV(\ell_{\yy,1},\ell_{\yy, 2},\ell_{\yy, 3}).    
\end{align*}
Note that $w_{a,b}$ with $a>3$ or $b>1$ is necessarily empty.

To describe the two dimensional variety $\VV(F)\cap \VV(g)$ we need to obtain the collection of witness points for $Y:=\VV(F)\cap \VV(g)$:
\begin{align*}
w_{2,0}(Y) :=
Y \cap \VV(\ell_{\xx,1},\ell_{\xx,2})
    ,\\
w_{1,1}(Y) := 
Y \cap \VV(\ell_{\xx,1},\ell_{\yy,1})
    ,\\
w_{0,2}(Y) :=
Y\cap  \VV(\ell_{\yy,1},\ell_{\yy,2}).
\end{align*}

To obtain $w_{a,b}(Y)$, 
for $a+b=2$, we run $u$-generation for the curve 
$\VV(F)\cap \VV(L)$ where $L=\{\ell_{\xx,1},\dots, \ell_{\xx,a},\, \ell_{\yy,1},\dots, \ell_{\yy,b}\}$ 
using start points obtained by  Algorithm~\ref{algo:bi-hom-start-points} with 
$(F\cup L, \ell_{a+1}, w_{a+1,b})$ and $(F\cup L,\ell_{b+1},w_{a,b+1})$ playing the role of $w_\xx$ and $w_\yy$.

In the previous section, $u$-generation and regeneration can both be used inside of the equation-by-equation cascade of
Algorithm~\ref{algo:cascade} based on intersection with a hypersurface in projective space.
The analogue in a product of projective spaces is known as \emph{multiregeneration}~\cite[Section 4]{MR4121336}.
A variant based on $u$-generation may be developed using the ideas outlined above.

\section{Computational experiments}\label{sec:experiments}

In this section, we report the computational results using our initial implementation of $u$-generation. 
For a fair comparison, we also implement regeneration in a similar fashion, using the same homotopy tracking toolkit provided by the package {\tt NumericalAlgebraicGeometry} \cite{leykin2011numerical} in \texttt{Macaulay2} \cite{M2www}. 
Default path-tracker settings for the method \texttt{trackHomotopy} are used, except where noted.
All computations were performed using a 2012 iMac with \SI{16}{GB}, working at \SI{1.6}{\GHz}.
Our modest aim is to convince the practitioner that 
$u$-generation is competitive as an equation-by-equation method, and that some initial successes of our implementation motivate further investigation and
more refined~heuristics.

\subsection{Intersecting with a hypersurface in projective space}\label{subsec:u-Pn}

The $u$-generation-based Algorithm~\ref{algo:intersection} for computing $X \cap \VV(g_1)$ requires tracking $\deg (g_1) \deg (X)$-many paths.
Regeneration also requires tracking this many paths, plus an additional $(\deg (g_1)-1) \deg (X)$ paths during the preparation stage.
In a rough analysis where we assume all paths have the same cost, we would expect that
\begin{equation}\label{eq:fraction-timing}
\displaystyle\frac{t_{u\text{-gen}}}{t_{\text{prep}}+t_{\text{regen}}} 
\approx 
\displaystyle\frac{\deg (g_1) \deg (X)}{2 \deg (g_1) \deg (X) - \deg (X)} 
=
\displaystyle\frac{1}{2-1/\deg(g_1)}
,
\end{equation}
where $t_{\text{prep}}, t_{\text{regen}},$ and $t_{\text{ugen}}$ denote the total time spent tracking paths during preparation, regeneration, or $u$-generation, respectively.
This analysis would then suggest that $u$-generation has
$\approx 33.3 \%$  savings over regeneration when $\deg (g_1)=2,$ and an asymptotic $50 \%$ savings when $\deg (g_1)$ is large.

The assumption that paths tracked during preparation cost the same as other paths does not hold in practice.
This is partly because homotopy functions involved in the preparation stage are simpler, and also due to the fact we do not expect any paths to diverge during this stage.
It is therefore worth investigating to what extent, if any, our proposed method may actually deliver any savings.
To that end, we considered two well-studied families of benchmark polynomial systems, and performed the following experiment for each:
\begin{itemize}
\item[1.] Drop an equation $g_1$ from the homogenized system, and compute a projective witness set for the projective curve $X$ defined by the remaining equations.
\item[2.] Run Algorithm~\ref{algo:intersection} to solve the original system.
\item[3.] Similarly to the previous step, use regeneration to solve the original system.
\end{itemize}
For all systems considered in this subsection, the witness set for $X$ is computed using a total-degree homotopy.
The timings we report for Steps 2 and 3 above do not reflect the full cost of solving these systems from scratch with equation-by-equation methods.
Nevertheless, this experiment is sufficient to make a meaningful comparison between regeneration and $u$-generation.

The first of the benchmark systems studied is the Katsura-$n$ family, which arose originally in the study of magnetism~\cite{katsura1990}.
For a given $n,$ this is a system of inhomogeneous equations in $n+1$ unknowns $x_0,\ldots , x_n$: writing $x_{-i} = x_i,$
\begin{align*}
\sum_{i=-n}^n x_i &= 1,\\
\sum_{i=-n}^n x_l x_{m-l} &= x_{m},\quad m\in \{-n+1,-n,..,n-1\} .
\end{align*}
For this family, the B\'{e}zout bound on the number of roots $2^{n-1}$ is tight.
In our experiments, the dropped equation $g_1$ is one of the $n$ quadratic equations.
We also considered the classic cyclic $n$-roots problem~\cite{DBLP:conf/issac/BackelinF91}:
\begin{align*}
\displaystyle\sum_{i=1}^n x_i \cdot x_{i+1} \cdots x_{(i+m) \% n} & =0, \quad m\in \{ 1, \ldots , n-1\},\\
x_0 \cdot x_1 \cdots x_{n-1}-1 &=0.
\end{align*}
This system has infinitely many solutions when $n$ is not squarefree.
For small, squarefree values of $n,$ this system has been observed to have all isolated solutions, their number matching the polyhedral bound of Bernstein's theorem~\cite{Bernstein}.
As a result, the polyhedral homotopy~\cite{HuStu95} and its various implementations~\cite{verschelde1999algorithm,HomotopyContinuationjulia,chen2014hom4ps} , is better-suited for this problem than equation-by-equation methods.
We include it in our experiments as a further point of comparison between regeneration and $u$-generation.
Based on the naive analysis, we might expect substantial savings when dropping the equation with $\deg (g_1) = n.$

\begin{figure}
\hspace{7.5em}
$\overbrace{\hspace{10.85em}}^{\text{regeneration}}$ $\overbrace{\hspace{10.25em}}^{\text{$u$-generation}}$\\
\resizebox{0.8\textwidth}{!}{\begin{tabular}{c|c||c|c|c||c|c|c}
system & $\#$ solutions & $\#$ paths & $t_{\text{prep}}$ & $t_{\text{regen}}$ & $\#$ paths & $t_{u\text{-gen}}$ & $\frac{t_{u\text{-gen}}}{t_{\text{prep}}+t_{\text{regen}}}$ \\[0.4em]
\hline 
katsura-8 &  \num{256} & \num{384} & \SI{1.1e-1}{\second} & \SI{2.5e-1}{\second}  & \num{256} & \SI{3.2e-1}{\second} & \textbf{.87}\\
katsura-9 &  \num{512} & \num{786} & \SI{2.7e-1}{\second} &\SI{6.0e-1}{\second}  & \num{512} & \SI{7.3e-1}{\second}  & \textbf{.84}\\
katsura-10 &  \num{1024} & \num{1536} & \SI{6.2e-1}{\second} & \SI{1.5}{\second} & \num{1024} & \SI{1.8}{\second} & \textbf{.86}\\
katsura-11 &  \num{2048} & \num{3072} & \SI{1.5}{\second} & \SI{3.2}{\second}  & \num{2048} & \SI{3.9 }{\second} & \textbf{.83}\\
katsura-12 &  \num{4096} & \num{6144} & \SI{3.6}{\second} & \SI{8.5}{\second} & \num{4096} & \SI{10.4}{\second} & \textbf{.85}\\
katsura-13 &  \num{8192} & \num{12288} & \SI{9.3}{\second}  & \SI{22.7}{\second}  & \num{8192}  &  \SI{26.9} & \textbf{.84} \\ 
cyclic-5 & \num{70} &  \num{126} & \SI{2.1e-2}{\second} & \SI{4.4e-2}{\second} & \num{70} & \SI{5.6e-2}{\second} & \textbf{.87}\\
cyclic-6 & \num{156} & \num{286} & \SI{5.7e-2}{\second} & \SI{1.4e-1}{\second}  & \num{156} & \SI{1.8e-1}{\second} &  \textbf{.90}\\
cyclic-7 & \num{924} & \num{1716} & \SI{4.5e-1}{\second} & \SI{1.4}{\second} & \num{924} & \SI{1.7}{\second} & \textbf{.89} 
\end{tabular}}
\caption{Timing Algorithm~\ref{algo:intersection} for benchmark polynomial systems.}\label{fig:benchmark-timings}
\end{figure}

In~\Cref{fig:benchmark-timings}, we display the results of our experiment on these benchmark problems.
Timings for each of the methods were averaged over $100$ iterations.
For cyclic-$7$, we observed multiple path failures for both methods with the default tracker settings, so we used a more permissive minimum stepsize of $10^{-8}$ for this case.
We observe in all cases that $u$-generation outperforms regeneration in terms of runtime.
Contrary to what the naive analysis would predict, typical savings are in the range of 10--15\% for both the Katsura and cyclic benchmarks.

To further investigate the nature of potential savings of $u$-generation over regeneration, we considered the following family of \demph{banded quadrics}.
Fixing integers $2 \le k \le n$, this is a homogeneous system $f_1,\ldots , f_n \in \CC [\xx ]$ where $f_1$ is linear and $f_2,\dots,f_n$ have the form 
\[
f_i (\xx) = \begin{bmatrix} x_i & \cdots & x_{(i +k\mod n)} \end{bmatrix}
\begin{bmatrix}
c_{i,1,1} & \cdots & c_{i,1,k}\\
\vdots & \ddots & \vdots \\
c_{i,k,1} & \cdots & c_{i,k,k} 
\end{bmatrix}
\begin{bmatrix} x_i \\ \vdots \\ x_{(i +k \mod n)} \end{bmatrix}.
\]
Here, the real and imaginary parts of the parameters $c_{i,j,k}$ are drawn randomly from the interval $[0,1]$.
We apply the same experiment from before to the homogeneous system, using a projective witness set for the projective curve $\VV(f_1,\dots,f_{n-1})$ 
to compute the intersection 
$\VV(f_1,\dots,f_{n-1})\cap \VV(f_n)$.
Timings for $n=12$ are given in~\Cref{fig:banded-quadrics-timings}.
\begin{figure}
\hspace{7.5em}
$\overbrace{\hspace{10.85em}}^{\text{regeneration}}$ $\overbrace{\hspace{10.25em}}^{\text{$u$-generation}}$\\
\resizebox{0.8\textwidth}{!}{\begin{tabular}{l|c||c|c|c||c|c|c}
$(n,k)$ & $\#$ solutions & $\#$ paths & $t_{\text{prep}}$ & $t_{\text{regen}}$ & $\#$ paths & $t_{u\text{-gen}}$ & $\frac{t_{u\text{-gen}}}{t_{\text{prep}}+t_{\text{regen}}}$ \\[0.4em]
\hline 
$(12, 2)$ &  \num{2048} & \num{3072} & \SI{8.9e-1}{\second} & \SI{3.0}{\second}  & \num{2048} & \SI{3.4}{\second} & \textbf{.87}\\  
$(12,3)$ &  \num{2048} & \num{3072} & \SI{1.3}{\second} &\SI{3.3}{\second}  & \num{2048} & \SI{3.9}{\second}  & \textbf{.87}\\
$(12,4)$ &  \num{2048} & \num{3072} & \SI{1.3}{\second} & \SI{3.7}{\second} & \num{2048} & \SI{4.5}{\second} & \textbf{.90}\\
$(12,5)$ &  \num{2048} & \num{3072} & \SI{1.5}{\second} & \SI{4.0}{\second}  & \num{2048} & \SI{5.0}{\second} & \textbf{.90}\\
$(12,6)$ &  \num{2048} & \num{3072} & \SI{1.8}{\second} & \SI{4.6}{\second} & \num{2048} & \SI{5.6}{\second} & \textbf{.89}\\
$(12,7)$ &  \num{2048} & \num{3072} & \SI{1.9}{\second}  & \SI{4.8}{\second}  & \num{2048}  &  \SI{6.1}{\second} & \textbf{.90} \\ 
$(12,8)$ & \num{2048} &  \num{3072} & \SI{2.2}{\second} & \SI{5.9}{\second} & \num{2048} & \SI{7.0}{\second} & \textbf{.90}\\
$(12,9)$ & \num{2048} &  \num{3072} & \SI{2.5}{\second} & \SI{6.1}{\second} & \num{2048} & \SI{7.7}{\second} & \textbf{.90}
\end{tabular}}
\caption{Timing Algorithm~\ref{algo:intersection} for banded quadrics with the number of variables fixed by $n=12.$
Each system is dehomogenized with a random chart, giving a system in $13$ unknowns.
}\label{fig:banded-quadrics-timings}
\end{figure}

One possible conclusion to draw from this experiment is that we may expect more savings from $u$-generation when the equations involved are more sparse. 
This is plausible for the case of banded quadrics when the value of $k$ is small in comparison to $n,$ in which case the cost of tracking a path in the preparation stage is more comparable to the cost for the paths in the second stage of regeneration.
On the other hand, we still observe in this experiment a steady 10 \% savings as $k$ approaches~$n.$

\subsection{Maximum likelihood estimation for matrices with rank constraints}\label{subsec:MLE}

Consider a probabilistic experiment where we flip two coins.
Each of the coins may be biased, but we model them with the same binary probability distribution.
Let $p_{1 1}$ denote the probability of obtaining two heads, $p_{2 2}$ the probability of two tails, and $p_{1 2} = p_{2 1}$ the probability of one head and one tail in either order.
The individual coin flips are independent if and only if
\begin{equation}
\label{eq:Pmatrix-2}
\rank \begin{bmatrix}
2 p_{1 1} & p_{1 2}\\
p_{1 2} & 2 p_{2 2}
\end{bmatrix} \le 1.
\end{equation}
Similarly, the independence of two identically distributed $n$-ary random variables is modeled by a $n \times n$ symmetric matrix of rank $1.$
More generally, rank constraints
\begin{equation}
\label{eq:Pmatrix-n}
\rank \begin{bsmallmatrix}
2p_{1 n} & p_{12} &\dots & p_{1 n}\\
p_{1 2} & 2p_{22} &  & p_{2n}\\
\vdots &         &\ddots & \vdots \\
p_{1 n} &  p_{2n}      & \dots      & 2p_{nn}
\end{bsmallmatrix} \le r
\end{equation}
must be satisfied for all points in the \demph{symmetric rank-constrained model}.
This statistical model is defined as the set of all points in the probability simplex
\[
\{ (p_{1 1}, \ldots , p_{n n}) \in \mathbb{R}^{(n^2+n)/2} \mid \displaystyle\sum_{1\le i \le j \le n} p_{i j} = 1, \, p_{i j} \ge 0 \} \]
which satisfy the rank constraints~\eqref{eq:Pmatrix-n}.
It is an algebraic relaxation of an associated $r$-th mixture model~\cite{HRS-rank-constraints,KRS15}.

If we are given a $n\times n$ symmetric matrix $U$ counting several samples from this statistical model, it is of interest to recover the underlying model parameters $( p_{i j} )_{1 \le i \le j \le n}.$
A popular approach in statistical inference is \demph{maximum likelihood estimation}.
In our case of interest, this means computing a global maximum of the likelihood function
\begin{equation}
\label{eq:likelihood}
\ell_U (p_{1 1}, \ldots , p_{n n}) \quad = \quad
\frac{\prod_{i\leq j} p_{ij}^{u_{ij}} }{\bigl(\,\sum_{i\leq j} p_{ij}\,\bigr)^{\sum_{i\leq j} u_{ij}}}
\end{equation}
restricted to the symmetric rank-constrained model.
Local optimization heuristics such as the EM algorithm or gradient ascent are popular methods for maximum-likelihood estimation, but are susceptible to local maxima.
For generic $U,$ the number of critical points of $\ell_U$ on the set of complex matrices satisfying~\eqref{eq:Pmatrix-n} is the \demph{ML degree} of the symmetric rank-constrained model.
For statistical models in general, the ML degree is an important measure of complexity and a well-studied quantity in algebraic statistics (see~\cite[Ch.~7]{SullivantBook}).

Previous work of Hauenstein, the third author, and Sturmfels~\cite{HRS-rank-constraints} demonstrates that homotopy continuation methods are a powerful technique for computing the critical points of $\ell_U$.
In~\cite[Theorem 2.3]{HRS-rank-constraints}, a table of ML degrees for the symmetric rank-constrained model for $n\le 6$ was obtained using the software \texttt{Bertini}~\cite{Bertini}.
In that table, ML degrees for $(n,r) \in \{ (6,2), (6,3) \}$ were excluded.
We provide the missing entries of that table in~\Cref{fig:ML-degrees}.

\newcolumntype{x}[1]{>{\centering\let\newline\\\arraybackslash\hspace{0pt}}p{#1}}
\begin{figure}
\[
\begin{array}{c|x{1cm}cccc}r\setminus n & 2 & 3 & 4 & 5 & 6\\ \hline
1 & 1 & 1 & 1 & 1 & 1\\
2 & 1 & 6 & 37 & 270& 2341\\
3 &  & 1 & 37& 1394& \mathbf{68774}\\
4 &  &  & 1& 270&  \mathbf{68774} \\
5 &  &  &  & 1& 2341\\
6 &  &  &  &  & 1
\end{array}
\]
\caption{ML degrees for the symmetric rank-constrained model, for $n\times n$ symmetric matrices of rank $\le r,$ for all $1\le r\le n\le 6$.}\label{fig:ML-degrees}
\end{figure}

The symmetry in each column of the table of ML degrees is explained by a duality result of Draisma and the third author~\cite[Theorem 11]{DR14},  establishing a remarkable bijection between the critical points for the rank-$r$ and rank-$(n-r+1)$ models.
Thus, to complete the table, it suffices to compute the ML degree for $(n,r)=(6,3),$ which we report to be \textbf{68774}.
We managed to verify the ML degrees in this table using several different techniques: equation-by-equation methods, monodromy~\cite{IMA-reference,MARTINDELCAMPO2017559}, and several combinations thereof.
We focus our discussion of computational results on the equation-by-equation methods, which we hope may serve as a useful starting point for future benchmarking studies.

To make computing the ML degrees more amenable to equation-by-equation approach, we use a \demph{symmetric local kernel formulation} of the problem, which is a square polynomial system in $\binom{n+1}{2}$ unknowns~\cite[Eq 2.13]{HRS-rank-constraints}.
The unknowns are the entries of three auxiliary matrices,
\small 
\[
P_1:= 
\begin{bsmallmatrix}
2p_{12} & p_{12} &\dots & p_{1r}\\
p_{12} & 2p_{22} &  & p_{2r}\\
\vdots &         &\ddots & \vdots \\
p_{1r} &  p_{2r}      & \dots      & 2p_{rr}
\end{bsmallmatrix},
\quad
L_1:=
\begin{bsmallmatrix}
\ell_{1,1} & \dots & \ell_{1,r}\\
\vdots & \ddots &\vdots \\
\ell_{n-r,1}&\dots & \ell_{n-r,r}
\end{bsmallmatrix},
\quad
\Lambda :=
\begin{bsmallmatrix}
\lambda_{11} & \lambda_{12} &\dots & \lambda_{1,n-r}\\
\lambda_{12} & \lambda_{22} &  & \lambda_{2,n-r}\\
\vdots &         &\ddots & \vdots \\
\lambda_{1,n-r} &  \lambda_{2,n-r}      & \dots      & \lambda_{r,n-r}
\end{bsmallmatrix}.
\]
\normalsize
The equations which make up this square system are the column sums and the entries above the diagonal in the $n\times n$ symmetric matrix 
\begin{equation}\label{eq:mle-matrix}
\begin{bsmallmatrix}
P_1 & P_1L_1^T\\
L_1P_1 & L_1P_1L_1^T
\end{bsmallmatrix}
 \odot
\begin{bsmallmatrix}
L_1^T\Lambda L_1 & \Lambda L_1\\
L_1\Lambda & \Lambda
\end{bsmallmatrix}
+ 
\sum_{i\leq j} u_{ij} \begin{bsmallmatrix}
P_1 & P_1L_1^T\\
L_1P_1 & L_1P_1L_1^T
\end{bsmallmatrix}
- 2I_n \odot U,
\end{equation}
where $\odot$ denotes the Hadamard product.
There are three natural variable groups given by the auxiliary matrices $P_1$, $L_1$, and $\Lambda$. 
Dropping a single equation $g_1$ from our square system gives equations which vanish on an affine patch of an irreducible curve
\[
X \subset \PP^{r \times r} \times \PP^{(n-r)\times r} \times \PP^{(n-r+1)\times (n-r)/2}.
\]
Slicing in each of the three projective factors, we may compute three sets of witness points $W_{P_1}, W_{L_1}, W_{\Lambda}$ for $X$ using monodromy.
These witness points can be used to compute start points for both regeneration and $u$-generation.
In our experiments, the dropped equation $g_1$ is the $(1,n)$-th entry of the matrix~\eqref{eq:mle-matrix}, which has degree $(1,2,1).$
Thus, there are $\# W_{P_1} + 2 \# W_{L_1} + \# W_{\Lambda_1}$ start points for both methods, and an additional $\# W_{L_1}$ paths must be tracked in the preparation phase for regeneration.
For $u$-generation, the start points were obtained using the heuristic proposed in Algorithm~\ref{algo:bi-hom-start-points} with $\epsilon = 10^{-5}.$
We do not eliminate the three $u$-variables corresponding to each factor---see~\Cref{rem:eliminate-u}.

We use two non-default path-tracker tolerances: a minimum stepsize of $10^{-14}$ by setting \texttt{tStepMin=>1e-14} and a maximum of $2$ Newton iterations for every predictor step by setting \texttt{maxCorrSteps=>2}.
The latter option is more conservative than the default of $\le 3$ Newton steps, which increases the chances of path-jumping.
With fewer corrector steps, a smaller timestep may be needed to track paths successfully.
We note that comparable tolerances are the defaults used in \texttt{Bertini}~\cite[Appendix~E.4.4,~E.4.7]{bertini-book}.

\Cref{fig:known-ML-timings} and~\cref{fig:ML-w-monodromy} show that the ML degrees are significantly smaller than the number of start points.
This implies that many endpoints will lie at the hyperplane at infinity in one of the projective factors.
We declare an endpoint to be finite when its three homogeneous coordinates exceed $10^{-6}$ in magnitude.

\begin{figure}
\hspace{5.95em}
$\overbrace{\hspace{11.1em}}^{\text{regeneration}}$ $\overbrace{\hspace{10.5em}}^{\text{$u$-generation}}$\\
\resizebox{0.8\textwidth}{!}{\begin{tabular}{c|c||c|c|c||c|c|c}
$(n,r)$ & ML degree &
$\# \text{paths}$
&
$t_{\text{prep}}$ & $t_{\text{regen}}$ &  
$\# \text{paths}$ &
 $t_{u\text{-gen}}$ &  $\frac{t_{u\text{-gen}}}{t_{\text{prep}}+t_{\text{regen}}}$\\
\hline 
(3,2) & 6 & 19 & \SI{3e-3}{\second}  & \SI{3e-2}{\second} &  15 &  \SI{5e-2}{\second}  &  \textbf{1.52} \\
(4,2) & 37 & 156 &  \SI{6e-2}{\second}   &   \SI{5e-1}{\second} & 118 &  \SI{6e-1}{\second} & \textbf{1.07} \\
(5,2) & 270 & \num{1313} & \SI{3e-1}{\second} & \SI{6}{\second}   & \num{980} &  \SI{9}{\second} & \textbf{1.43} \\
(5,3) & \num{1394} & \num{6559} & \SI{5}{\second} &  \SI{55}{\second} &  \num{5185}  & \SI{66}{\second} & \textbf{1.10}\\
(6,2) & \num{2341} & \num{9137} & \SI{8}{\second} & \SI{94}{\second} &   \num{12343} & \SI{121}{\second} & \textbf{1.18}
\end{tabular}}
\caption{Timings of regeneration and $u$-generation for computing previously-known ML degrees.}\label{fig:known-ML-timings}
\end{figure}

In~\Cref{fig:known-ML-timings}, we list timings for computing the previously-known ML degrees.
We see that the number of additional paths required by regeneration is rather small in comparison with the ML degree.
Thus, we should not expect to see much of an advantage of $u$-generation over regeneration.
Indeed, although the timings of both methods are competitive, regeneration is the clear winner.
Although $t_{\text{regen}}$ and $t_{u-\text{gen}}$ time the same number of path-tracks, we see that $t_{u\text{-gen}}$ is consistently larger.
One explanation is that $u$-generation uses an extra three variables here, increasing costs associated with repeated function evaluation and numerical linear algebra.
\begin{remark}
The heuristic for computing start points in Algorithm~\ref{algo:bi-hom-start-points} imposes an additional overhead for the multiprojective $u$-generation. This   
routine is implemented in the top-level language of \texttt{Macaulay2} and currently takes a significant portion of the runtime. 
Since this heuristic may be improved in future work, we didn't pursue a low-level implementation,   
which would dramatically reduce the runtime of this routine making it negligible in comparison to the other parts. 
\end{remark}

\begin{figure}
\resizebox{0.8\textwidth}{!}{\begin{tabular}{c|c|c|c|c|c|c|c|c|c}
method & $t_1$ & \# collected  & $t_2$ & \# collected  & $t_3$ & \# collected  & $t_4$ & \# collected & $t_{\text{total}}$\\
\hline
regeneration & \SI{12498}{\second} & \num{68767} & \SI{12666}{\second} & \num{68773} & \SI{20210}{\second} & \num{68773} & \SI{20393}{\second} & \rootcount & \SI{65767}{\second}\\
$u$-generation & \SI{13723}{\second} & \num{68729} & \SI{13900}{\second} & \num{68773} & \SI{25433}{\second} & \rootcount & \SI{13689}{\second} &  \rootcount & \SI{66745}{\second}
\end{tabular}}
\caption{Timings $t_1,t_2,t_3,t_4$ of four runs of regeneration and $u$-generation for computing the ML degree of \textbf{68774} for $(n,r)=(6,3).$}\label{fig:unknown-ML-timings}.
\end{figure}

For the previously-unknown ML degree when $(n,r)=(6,3),$ we conducted an experiment where we ran each of the two equation-by-equation methods four times and used the \emph{union} of all finite endpoints to estimate the ML degree.
The four runs are not identical because the homotopies are randomized using the $\gamma$-trick.
Timings are shown in~\Cref{fig:unknown-ML-timings}.
The variability of the timings is not so surprising, since paths are randomized according to $\gamma$-type tricks and also due to the permissive step-size of $10^{-14}$; if there are many ill-conditioned paths, the adaptively-chosen stepsize for each of them may shrink and remain small throughout path-tracking.
Ultimately, both methods yield the same root count, which $u$-generation attains by the third iteration.
We also note that the total timing for regeneration and $u$-generation turn out to be very close, though this might well be an anomaly.

\begin{figure}
\hspace{5.em}
$\overbrace{\hspace{13.25em}}^{\text{regeneration}}$ $\overbrace{\hspace{10.15em}}^{\text{$u$-generation}}$\\
\resizebox{0.8\textwidth}{!}{\begin{tabular}{c|c||c|c|c|c|c||c|c|c|c}
$(n,r)$ & ML degree &
$\# \text{paths}$
&
$t_{\text{prep}}$ & $t_{\text{regen}}$ & $t_{\text{mon}}$ & $t_{\text{total}}$ &
$\# \text{paths}$ &
 $t_{u\text{-gen}}$ & $t_{\text{mon}}$ & $t_{\text{total}}$\\
\hline 
(5,2) & 270 & \num{1313} & \SI{7e-1}{\second} & \SI{5}{\second} & \SI{9}{\second} & \SI{14}{\second} & \num{980} &  \SI{9}{\second} & \SI{10}{\second} & \SI{19}{\second} \\
(5,3) & \num{1394} & \num{6559} & \SI{4}{\second} & \SI{47}{\second} & \SI{34}{\second} & \SI{85}{\second} & \num{5185} &  \SI{61}{\second} & \SI{52}{\second} & \SI{113}{\second}\\
(6,2) & \num{2341} & \num{9137} & \SI{7}{\second} & \SI{88}{\second} & \SI{124}{\second} & \SI{219}{\second} & \num{12343} & \SI{110}{\second} & \SI{109}{\second} & \SI{219}{\second}\\
(6,3) & \num{68774} & \num{387114}& \SI{1136}{\second} & \SI{11822}{\second} & \SI{6557}{\second} & \SI{19515}{\second} & \num{301675} &  \SI{13883}{\second} & \SI{6509}{\second} & \SI{20392}{\second}
\end{tabular}}
\caption{Regeneration and $u$-generation followed by a monodromy loop.}\label{fig:ML-w-monodromy}
\end{figure}

One might reasonably be concerned that multiple runs of both equation-by-equation methods were needed to compute the ML degree for $(n,r)=(6,3).$
In theory, both are probability-one methods.
In practice, most implementations of homotopy continuation will miss some solutions when presented with a sufficiently difficult problem.
One well-known practical issue is path-jumping, which may lead in some cases to duplicate endpoints.
The existence of many solutions at infinity, which are often highly-singular, presents another obstacle.
This obstacle would be even more of a concern if we were to consider other homotopy continuation methods.
We note, for instance, the number of paths tracked by the celebrated polyhedral homotopy may be prohibitive for the $(n,r)=(6,3)$ case.
Using~\cite{HomotopyContinuationjulia} and~\cite{gfan}, which both implement the mixed volume algorithm described in~\cite{jensen2016tropical}, we determined that the polyhedral root count for the symmetric local kernel formulation is \textbf{27174865}---three orders of magnitude over the ML degree, and two over the number of paths tracked in our experiments.

\Cref{fig:ML-w-monodromy} displays the results of another experiment where we ran each equation-by-equation method \emph{exactly once} and collected any missed solutions using a monodromy loop.
A similar strategy for dealing with failed solutions is outlined in~\cite[Sec.~3.2]{sparse-JSAG}.
This experiment gives us more confidence in the reported root count of \rootcount.
For $(n,r)=(6,3)$ the total runtime compares favorably to~\Cref{fig:unknown-ML-timings}, suggesting another strategy the practitioner may keep in mind.
We point out that such strategies may be of interest in numerical irreducible decomposition, where the \demph{monodromy breakup algorithm}~\cite{monodromy-breakup} is run following the cascade of Algorithm~\ref{algo:cascade}.

We close with a concrete numerical example of maximum-likelihood estimation, where the target system and its solutions resulting from $u$-generation may now serve as the start-system for a parameter homotopy~\cite[Ch.~7]{SW05}.

\begin{example}\label{example:mLE}
Consider the matrix of counts
\[
U = \begin{bsmallmatrix}
       2&4&2&4&4&4\\
       4&2&2&6&6&6\\
       2&2&1&2&2&2\\
       4&6&2&3&6&6\\
       4&6&2&6&3&6\\
       4&6&2&6&6&3\\
       \end{bsmallmatrix}.
\]
Let $P_{\text{emp}}$ denote the empirical $P$-matrix
\[
P_{\text{emp}} = \displaystyle\frac{1}{\sum_{i\leq j} u_{ij}}  \begin{bsmallmatrix}
2 u_{11} & u_{1 2} & \cdots & u_{1 6} \\
u_{1 2} &2 u_{2 2} & \cdots & u_{2 6}\\
\vdots & \vdots & \ddots & \vdots \\
u_{1 6} & u_{2 6} & \cdots & 2u_{6 6}
       \end{bsmallmatrix}
\]
This matrix has rank $4.$
To compute a maximum-likelihood estimate given $U,$ we use a parameter homotopy with \rootcount~start points furnished by $u$-generation.
Among the approximate endpoints of this parameter homotopy are 1082 labeled as failed paths, which have large and non-real coordinates.
Among the successful endpoints, there are 4108 real solutions.
However, only three are statistically valid in the sense that the rank-constrained $P$-matrix has non-negative entries. This matrix may be recovered with the formula
\[
P = \begin{bsmallmatrix}
P_1 & P_1L_1^T\\
L_1P_1 & L_1P_1L_1^T
\end{bsmallmatrix}.
\]
We list the $P$-matrices and likelihood values for the three statistically valid critical points, along with the empirical $P$-matrix (which does not lie on the rank-constrained model):
\begin{align*}
P^{(1)} &\approx \left(\begin{smallmatrix}
       {.046}&{.053}&{.03}&{.054}&{.054}&{.054}\\
       {.053}&{.053}&{.026}&{.079}&{.079}&{.079}\\
       {.030}&{.026}&{.024}&{.026}&{.026}&{.026}\\
       {.054}&{.079}&{.026}&{.079}&{.079}&{.079}\\
       {.054}&{.079}&{.026}&{.079}&{.079}&{.079}\\
       {.054}&{.079}&{.026}&{.079}&{.079}&{.079}\\
       \end{smallmatrix}\right), & \log \ell_U (P^{(1)}) = -223.264,\\
P^{(2)} &\approx \left(\begin{smallmatrix}
       {.053}&{.051}&{.026}&{.053}&{.053}&{.053}\\
       {.051}&{.069}&{.026}&{.074}&{.074}&{.074}\\
       {.026}&{.026}&{.026}&{.026}&{.026}&{.026}\\
       {.053}&{.074}&{.026}&{.080}&{.080}&{.080}\\
       {.053}&{.074}&{.026}&{.080}&{.080}&{.080}\\
       {.053}&{.074}&{.026}&{.080}&{.080}&{.080}\\
       \end{smallmatrix}\right), & \log \ell_U (P^{(2)}) =-222.979,\\
P^{(3)} &\approx        \left(\begin{smallmatrix}
       {.048}&{.054}&{.020}&{.056}&{.056}&{.056}\\
       {.054}&{.052}&{.029}&{.077}&{.077}&{.077}\\
       {.020}&{.029}&{.014}&{.032}&{.032}&{.032}\\
       {.056}&{.077}&{.032}&{.077}&{.077}&{.077}\\
       {.056}&{.077}&{.032}&{.077}&{.077}&{.077}\\
       {.056}&{.077}&{.032}&{.077}&{.077}&{.077}\\
       \end{smallmatrix}\right), & \log \ell_U (P^{(3)}) = -222.901,\\
P_{\text{emp}}
&\approx 
\left(\begin{smallmatrix}
       {.053}&{.053}&{.026}&{.053}&{.053}&{.053}\\
       {.053}&{.053}&{.026}&{.079}&{.079}&{.079}\\
       {.026}&{.026}&{.026}&{.026}&{.026}&{.026}\\
       {.053}&{.079}&{.026}&{.079}&{.079}&{.079}\\
       {.053}&{.079}&{.026}&{.079}&{.079}&{.079}\\
       {.053}&{.079}&{.026}&{.079}&{.079}&{.079}\\
       \end{smallmatrix}\right) ,
       & \log \ell_U (P_{\text{e}}) = -222.860.
      \end{align*}
The matrix $P^{(3)}$ is a good candidate for the maximum likelihood estimate.
\end{example}

\section{Conclusion}\label{sec:conclusion}

We propose $u$-generation as a novel equation-by-equation homotopy continuation method for solving polynomial systems.
The main theoretical merits of this method are that it works in both projective and multiprojective settings and requires tracking fewer total paths than regeneration.
Furthermore, our setup is easily implementable.
Computational experiments show that the method is promising and can be used to solve nontrivial problems, although we do not observe a decisive advantage of $u$-generation over regeneration.
More specifically, our proof-of-concept implementation achieves modest, but notable, savings in terms of computational time on several examples.
Furthermore, we show that multiprojective $u$-generation is a viable method for solving structured polynomial systems arising in maximum-likelihood estimation, including one previously-unsolved case.

As a future direction, we point out that more thorough experimentation with heuristics such as Algorithm~\ref{algo:bi-hom-start-points} might lead to a more robust implementation of $u$-generation.
Further exploration in both projective and multiprojective settings is warranted, particularly in the context of numerical irreducible decomposition, where equation-by-equation methods play a vital role.
\bibliographystyle{abbrv}
\bibliography{bib}

\begin{thebibliography}{10}

\bibitem{DBLP:conf/issac/BackelinF91}
J.~Backelin and R.~Fr{\"{o}}berg.
\newblock How we proved that there are exactly 924 cyclic 7-roots.
\newblock In S.~M. Watt, editor, {\em Proceedings of the 1991 International
  Symposium on Symbolic and Algebraic Computation, {ISSAC} '91, Bonn, Germany,
  July 15-17, 1991}, pages 103--111. {ACM}, 1991.

\bibitem{Bertini}
D.~J. Bates, J.~D. Hauenstein, A.~J. Sommese, and C.~W. Wampler.
\newblock {Bertini: Software for Numerical Algebraic Geometry}.
\newblock Available at bertini.nd.edu with permanent doi:
  dx.doi.org/10.7274/R0H41PB5.

\bibitem{bertini-book}
D.~J. Bates, J.~D. Hauenstein, A.~J. Sommese, and C.~W. Wampler.
\newblock {\em Numerically solving polynomial systems with {B}ertini},
  volume~25 of {\em Software, Environments, and Tools}.
\newblock Society for Industrial and Applied Mathematics (SIAM), Philadelphia,
  PA, 2013.

\bibitem{Bernstein}
D.~N. Bernstein.
\newblock The number of roots of a system of equations.
\newblock {\em Funkcional. Anal. i Prilo\v{z}en.}, 9(3):1--4, 1975.

\bibitem{HomotopyContinuationjulia}
P.~Breiding and S.~Timme.
\newblock Homotopycontinuation.jl: A package for homotopy continuation in
  {J}ulia.
\newblock In {\em Mathematical Software -- ICMS 2018}, pages 458--465. Springer
  International Publishing, 2018.

\bibitem{sparse-JSAG}
T.~Brysiewicz, J.~I. Rodriguez, F.~Sottile, and T.~Yahl.
\newblock Decomposable sparse polynomial systems.
\newblock {\em J. Softw. Algebra Geom.}, 11(1):53--59, 2021.

\bibitem{chen2014hom4ps}
T.~Chen, T.-L. Lee, and T.-Y. Li.
\newblock Hom4ps-3: a parallel numerical solver for systems of polynomial
  equations based on polyhedral homotopy continuation methods.
\newblock In {\em International Congress on Mathematical Software}, pages
  183--190. Springer, 2014.

\bibitem{DR14}
J.~Draisma and J.~Rodriguez.
\newblock Maximum likelihood duality for determinantal varieties.
\newblock {\em Int. Math. Res. Not. IMRN}, (20):5648--5666, 2014.

\bibitem{IMA-reference}
T.~Duff, C.~Hill, A.~Jensen, K.~Lee, A.~Leykin, and J.~Sommars.
\newblock Solving polynomial systems via homotopy continuation and monodromy.
\newblock {\em IMA J. Numer. Anal.}, 39(3):1421--1446, 2019.

\bibitem{M2www}
D.~R. Grayson and M.~E. Stillman.
\newblock Macaulay2, a software system for research in algebraic geometry.
\newblock Available at http://www.math.uiuc.edu/Macaulay2/.

\bibitem{HRS-rank-constraints}
J.~Hauenstein, J.~I. Rodriguez, and B.~Sturmfels.
\newblock Maximum likelihood for matrices with rank constraints.
\newblock {\em J. Algebr. Stat.}, 5(1):18--38, 2014.

\bibitem{MR4166467}
J.~D. Hauenstein, A.~Leykin, J.~I. Rodriguez, and F.~Sottile.
\newblock A numerical toolkit for multiprojective varieties.
\newblock {\em Math. Comp.}, 90(327):413--440, 2021.

\bibitem{MR4121336}
J.~D. Hauenstein and J.~I. Rodriguez.
\newblock Multiprojective witness sets and a trace test.
\newblock {\em Adv. Geom.}, 20(3):297--318, 2020.

\bibitem{MR2733776}
J.~D. Hauenstein and A.~J. Sommese.
\newblock Witness sets of projections.
\newblock {\em Appl. Math. Comput.}, 217(7):3349--3354, 2010.

\bibitem{HSW:regeneration}
J.~D. Hauenstein, A.~J. Sommese, and C.~W. Wampler.
\newblock Regeneration homotopies for solving systems of polynomials.
\newblock {\em Math. Comp.}, 80(273):345--377, 2011.

\bibitem{HSW:regenerative-cascade}
J.~D. Hauenstein, A.~J. Sommese, and C.~W. Wampler.
\newblock Regenerative cascade homotopies for solving polynomial systems.
\newblock {\em Appl. Math. Comput.}, 218(4):1240--1246, 2011.

\bibitem{HW:unification}
J.~D. Hauenstein and C.~W. Wampler.
\newblock Unification and extension of intersection algorithms in numerical
  algebraic geometry.
\newblock {\em Appl. Math. Comput.}, 293:226--243, 2017.

\bibitem{HuStu95}
B.~Huber and B.~Sturmfels.
\newblock A polyhedral method for solving sparse polynomial systems.
\newblock {\em Math. Comp.}, 64(212):1541--1555, 1995.

\bibitem{gfan}
A.~N. Jensen.
\newblock {G}fan, a software system for {G}r{\"o}bner fans and tropical
  varieties.
\newblock Available at
  \url{http://home.imf.au.dk/jensen/software/gfan/gfan.html}.

\bibitem{jensen2016tropical}
A.~N. Jensen.
\newblock Tropical homotopy continuation.
\newblock {\em arXiv preprint arXiv:1601.02818}, 2016.

\bibitem{katsura1990}
S.~Katsura.
\newblock Spin glass problem by the method of integral equation of the
  effective field.
\newblock {\em New Trends in Magnetism}, pages 110--121, 1990.

\bibitem{KRS15}
K.~Kubjas, E.~Robeva, and B.~Sturmfels.
\newblock Fixed points {EM} algorithm and nonnegative rank boundaries.
\newblock {\em Ann. Statist.}, 43(1):422--461, 2015.

\bibitem{leykin2011numerical}
A.~Leykin.
\newblock Numerical algebraic geometry.
\newblock {\em Journal of Software for Algebra and Geometry}, 3(1):5--10, 2011.

\bibitem{LVZ-deflation}
A.~Leykin, J.~Verschelde, and A.~Zhao.
\newblock Higher-order deflation for polynomial systems with isolated singular
  solutions.
\newblock In {\em Algorithms in algebraic geometry}, volume 146 of {\em IMA
  Vol. Math. Appl.}, pages 79--97. Springer, New York, 2008.

\bibitem{MARTINDELCAMPO2017559}
A.~Mart\'{\i}n~del Campo and J.~I. Rodriguez.
\newblock Critical points via monodromy and local methods.
\newblock {\em J. Symbolic Comput.}, 79(part 3):559--574, 2017.

\bibitem{SV:generic-points-on-algebraic-sets}
A.~J. Sommese and J.~Verschelde.
\newblock Numerical homotopies to compute generic points on positive
  dimensional algebraic sets.
\newblock {\em journal of complexity}, 16(3):572--602, 2000.

\bibitem{monodromy-breakup}
A.~J. Sommese, J.~Verschelde, and C.~W. Wampler.
\newblock Symmetric functions applied to decomposing solution sets of
  polynomial systems.
\newblock {\em SIAM J. Numer. Anal.}, 40(6):2026--2046 (2003), 2002.

\bibitem{SVW:equation-by-equation}
A.~J. Sommese, J.~Verschelde, and C.~W. Wampler.
\newblock Solving polynomial systems equation by equation.
\newblock In {\em Algorithms in algebraic geometry}, volume 146 of {\em IMA
  Vol. Math. Appl.}, pages 133--152. Springer, New York, 2008.

\bibitem{SW05}
A.~J. Sommese and C.~W. Wampler, II.
\newblock {\em The numerical solution of systems of polynomials}.
\newblock World Scientific Publishing Co. Pte. Ltd., Hackensack, NJ, 2005.
\newblock Arising in engineering and science.

\bibitem{Sottile2020}
F.~Sottile.
\newblock General witness sets for numerical algebraic geometry.
\newblock In {\em Proceedings of the 45th International Symposium on Symbolic
  and Algebraic Computation}, ISSAC '20, page 418–425, New York, NY, USA,
  2020. Association for Computing Machinery.

\bibitem{SullivantBook}
S.~Sullivant.
\newblock {\em Algebraic statistics}, volume 194 of {\em Graduate Studies in
  Mathematics}.
\newblock American Mathematical Society, Providence, RI, 2018.

\bibitem{verschelde1999algorithm}
J.~{Verschelde}.
\newblock Algorithm 795: Phcpack: a general-purpose solver for polynomial
  systems by homotopy continuation.
\newblock {\em ACM Transactions on Mathematical Software}, 25(2):251--276,
  1999.

\end{thebibliography}

\end{document}